\documentclass[a4paper,12pt,reqno]{amsart}
\usepackage{latexsym,amscd,amssymb,url}
\usepackage{xcolor}
\definecolor{dblue}{rgb}{0,0,.6}
\usepackage[colorlinks=true, linkcolor=dblue, citecolor=dblue, filecolor = dblue, menucolor = dblue, urlcolor = dblue]{hyperref}

\usepackage[all,cmtip]{xy}  
\usepackage{graphicx}
\usepackage{color}
\usepackage{hyperref}
\usepackage{enumerate}

\newtheorem{theorem}{Theorem}[section]

\theoremstyle{plain}

\newtheorem{conjecture}[theorem]{Conjecture}
\newtheorem{corollary}[theorem]{Corollary}

\newtheorem{definition}[theorem]{Definition}

\newtheorem{lemma}[theorem]{Lemma}

\newtheorem{proposition}[theorem]{Proposition}
\newtheorem{remark}[theorem]{Remark}

\newcommand{\Z}{\mathbb Z}
\newcommand{\Q}{\mathbb Q}

\newcommand{\C}{\mathbb C}
\newcommand{\N}{\mathbb N}

\newcommand{\CP}{\mathbb P}

\newcommand{\Hom}{\operatorname{Hom}}
\newcommand{\Pic}{\operatorname{Pic}}
\newcommand{\Div}{\operatorname{Div}}

\newcommand{\Aut}{\operatorname{Aut}}

\newcommand{\id}{\operatorname{id}}

\newcommand{\GL}{\operatorname{GL}}
\newcommand{\Spec}{\operatorname{Spec}}

\newcommand{\Alb}{\operatorname{Alb}}

\newcommand{\red}{\operatorname{red}}

\newcommand{\Jac}{\operatorname{Jac}}

\newcommand{\sm}{\operatorname{sm}}

\newcommand{\dashedlongrightarrow}{\xymatrix@1@=15pt{\ar@{-->}[r]&}}
\renewcommand{\longrightarrow}{\xymatrix@1@=15pt{\ar[r]&}}
\renewcommand{\mapsto}{\xymatrix@1@=15pt{\ar@{|->}[r]&}}
\renewcommand{\twoheadrightarrow}{\xymatrix@1@=15pt{\ar@{->>}[r]&}}
\newcommand{\hooklongrightarrow}{\xymatrix@1@=15pt{\ar@{^(->}[r]&}}
\newcommand{\congpf}{\xymatrix@1@=15pt{\ar[r]^-\sim&}}
\renewcommand{\cong}{\simeq}

\begin{document}   

\title[Holomorphic one-forms without zeros on threefolds]{Holomorphic one-forms without zeros on threefolds}

\author{Feng Hao}

\address{Mathematisches Institut, LMU M\"unchen, Theresienstr. 39, 80333 M\"unchen, Germany}

\curraddr{
KU Leuven, Celestijnenlaan 200B, B-3001 Leuven, Belgium 
}
\email{feng.hao@kuleuven.be}

\author{Stefan Schreieder}

\address{Mathematisches Institut, LMU M\"unchen, Theresienstr. 39, 80333 M\"unchen, Germany}
\email{schreieder@math.lmu.de}

\date{December 26th, 2019}
 


\begin{abstract}   
We show  that a smooth complex projective threefold admits a holomorphic one-form without zeros if and only if the underlying real $6$-manifold is a $C^\infty$-fibre bundle over the circle, and we give a complete classification of all threefolds with that property.
Our results prove a conjecture of Kotschick in dimension three. 
\end{abstract}

\maketitle

\section{Introduction}
 
\subsection{Holomorphic one-forms and fibre bundles over the circle}
For a smooth complex projective variety $X$ we may consider the following conditions:
\begin{enumerate}[(A)]
\item $X$ admits a holomorphic one-form without zeros; 
\label{item:holo}
\item 
\label{item:closed}
$X$ admits a real closed $1$-form without zeros;
or, by Tischler's theorem \cite{Ti} equivalently, the underlying differentiable manifold is a  $C^\infty$-fibre bundle over the circle.
\end{enumerate}
Note that while (\ref{item:holo}) is an algebraic condition on $X$, condition (\ref{item:closed}) is a differential geometric one which by \cite{Ti} characterizes the smooth manifold which underlies $X$ as $[0,1]\times M/\sim$, where $M$ is a closed manifold of real dimension $2\dim(X)-1$ and where $0\times M$ is identified with $1\times M$ via some diffeomorphism of $M$.

While  (\ref{item:holo}) $\Rightarrow$ (\ref{item:closed})  is clear, Kotschick conjectured \cite{Ko} that both condition might be equivalent to each other.
In  \cite{Sch19}, the second author developed an approach to this conjecture, showing that (\ref{item:closed}) implies 
\begin{enumerate}[(A)]
\setcounter{enumi}{2}
\item \label{item:exact} there is a holomorphic one-form $\omega \in H^0(X,\Omega_X^1)$ such that for any finite \'etale cover $\tau:X'\to X$, the sequence 
$$  H^{i-1}(X',\C)\stackrel{\wedge \omega'}\longrightarrow H^{i}(X',\C)\stackrel{\wedge \omega'}\longrightarrow H^{i+1}(X',\C) 
$$
given by cup product with $\omega':=\tau^\ast \omega$
is exact for all $i$.
\end{enumerate}
Moreover, all three conditions above coincide in dimension two \cite[Theorem 1.3]{Sch19}.
In this paper, we address the much more difficult case of threefolds. 

\begin{theorem} \label{thm:A=B=C}
Let $X$ be a smooth complex projective threefold.  
Then all three conditions above are equivalent to each other: (\ref{item:holo}) $ \Leftrightarrow$ (\ref{item:closed}) $ \Leftrightarrow$ (\ref{item:exact}).
\end{theorem}

In the appendix to this paper, we explain how the argument in \cite{Sch19} can be generalized to show that (\ref{item:exact}) is also implied by the following weak version of (\ref{item:closed}):
\begin{enumerate}[(A')]
\setcounter{enumi}{1} 
\item $X$ is homotopy equivalent to a CW complex $Y$ that admits a continuous map $f:Y\to S^1$ to the circle which is a finite $\Q$-homology fibration, i.e.\ $R^ if_\ast \Q$ are local systems of finite-dimensional $\Q$-vector spaces for all $i$.\label{item:B'}
\end{enumerate}
Since (\ref{item:closed}) $\Rightarrow$ (\ref{item:B'}') is clear (c.f.\ \cite{Ti}), this implies by Theorem \ref{thm:A=B=C} the following:

\begin{corollary} \label{cor:A=B'}
Let $X$ be a smooth complex projective threefold.  
Then (\ref{item:holo}) $ \Leftrightarrow$ (\ref{item:B'}').
In particular, the question whether $X$ carries a holomorphic one-form without zeros depends only on the homotopy type of $X$.
\end{corollary}

\subsection{Classifying threefolds whose underlying $6$-manifolds fibre over the circle}
Theorem \ref{thm:A=B=C} follows from the following strong classification result. 
\begin{theorem} \label{thm:classification}
If $X$ is a smooth complex projective threefold, any of the conditions (\ref{item:holo}), (\ref{item:closed}), (\ref{item:B'}') and (\ref{item:exact}) is equivalent to the following:
\begin{enumerate}[(A)] 
\setcounter{enumi}{3}
\item the minimal model program for $X$ yields a birational morphism  $\sigma:X\to X^{\min}$ to a smooth projective threefold $X^{\min}$, such that: \label{item:classification:MMP}
\begin{enumerate}[(1)] 
\item $\sigma:X\to X^{\min}$ is a sequence of blow-ups along smooth elliptic curves that are not contracted via the natural map to the Albanese variety $\Alb(X^{\min})$. \label{item:thm:classification:sigma}
\item There is a smooth morphism $\pi:X^{\min}\to A$ to a positive-dimensional abelian variety $A$. \label{item:thm:classification:pi}
\item \label{item:thm:classification:kappa=0,1,2,3}  
If $\kappa(X)\geq 0$, then a finite \'etale cover $\tau:X'\to X^{\min}$ splits into a product $X'\cong A'\times S'$,
 where $S'$ is smooth projective and $A'$ is an abelian variety such that for all $s'\in S'$, the natural composition
 $$
 A'\cong A'\times \{s'\}\hookrightarrow A'\times S'\cong X'\stackrel{\tau}\longrightarrow X^{\min}\stackrel{\pi}\longrightarrow A
 $$
 is a finite \'etale cover.

 
\item If $\kappa(X)=-\infty$, then one of the following holds: \label{item:thm:classification:kappa=-infty}
\begin{enumerate}[(i)]
\item $X^{\min}$ admits a smooth del Pezzo fibration over an elliptic curve;
\label{item:thm:classification:kappa=-infty:delPezzo}
\item $X^{\min}$ has the structure of a conic bundle $f:X^{min}\to S$ over a smooth projective surface $S$ which satisfies (\ref{item:holo}) $ \Leftrightarrow$ (\ref{item:closed}) $ \Leftrightarrow$ (\ref{item:exact}).  
Moreover, $f$ is either smooth, or $A$ is an elliptic curve and the degeneration locus of $f$ is a disjoint union of smooth elliptic curves on $S$ which are \'etale over $A$ (via the map $S\to A$ induced by $\pi$). 
\label{item:thm:classification:kappa=-infty:conic}
\end{enumerate}
\end{enumerate} 
\end{enumerate}
\end{theorem}

Note that  (\ref{item:classification:MMP}) $\Rightarrow$ (\ref{item:holo}), because a general one-form on $X$ has no zeros by (\ref{item:thm:classification:sigma}) and (\ref{item:thm:classification:pi}). 
Since (\ref{item:holo}) $\Rightarrow$ (\ref{item:closed}) $\Rightarrow$ (\ref{item:B'}') is clear and (\ref{item:B'}') $ \Rightarrow$ (\ref{item:exact}) is known (see \cite{Sch19} and Theorem \ref{thm:B'} in the Appendix), in order to prove Theorems \ref{thm:A=B=C} and \ref{thm:classification}, it thus suffices to show (\ref{item:exact}) $\Rightarrow$ (\ref{item:classification:MMP}).
In the course of our proof, we will obtain the following refined version of  (\ref{item:holo}) $ \Leftrightarrow$ (\ref{item:exact}).

\begin{theorem} \label{thm:cup_omega_exact}
Let $X$ be a smooth complex projective threefold and let $\omega\in H^0(X,\Omega_X^1)$ be a holomorphic one-form on $X$.
Then the following are equivalent:
\begin{enumerate}[(1)]
\item $\omega$ has no zero on $X$; \label{item:no-zero:thm:cup}
\item for any \'etale cover $\tau:X'\to X$, the sequence given by cup product with $\omega':=\tau^\ast \omega$
$$  
H^{i-1}(X',\C)\stackrel{\wedge \omega'}\longrightarrow H^{i}(X',\C)\stackrel{\wedge \omega'}\longrightarrow H^{i+1}(X',\C) 
$$
is exact for all $i$. 
\label{item:exact:thm:cup}
\end{enumerate} 
\end{theorem}

The implication (\ref{item:no-zero:thm:cup}) $\Rightarrow$ (\ref{item:exact:thm:cup}) is a result of Green and Lazarsfeld which holds in arbitrary dimensions, see \cite[Proposition 3.4]{GL}.
The converse implication (\ref{item:exact:thm:cup}) $\Rightarrow$  (\ref{item:no-zero:thm:cup}) has previously been proven in dimension two by the second author in \cite[Theorem 1.3]{Sch19}.

Theorem \ref{thm:cup_omega_exact} has the following interesting consequence.

\begin{corollary} \label{cor:cup_omega_exact:1}
For a holomorphic one-form $\omega \in H^0(X,\Omega_X^1)$ on a smooth complex projective threefold $X$, the condition that $\omega$ has no zeros on $X$ is a topological one which depends only on the cohomology class $[\omega]\in H^\ast(X,\C)$ of $\omega$ and the homotopy type of $X$.
\end{corollary}

By the above corollary, if $\omega$ is a holomorphic one-form without zeros on a smooth projective threefold $X$, then for any smooth projective threefold $X'$ which is deformation equivalent to $X$, and for any one-form $\omega'\in H^0(X',\Omega_{X'}^1)$ that is obtained via parallel transport of $\omega$ with respect to some path, $\omega'$ has no zeros on $X'$. 
This is interesting already in the case where $X=X'$.

\subsection{Around a theorem of Popa and Schnell}
Recall that items (\ref{item:thm:classification:sigma}) and  (\ref{item:thm:classification:pi}) imply condition (\ref{item:holo}).
Hence, Theorem \ref{thm:classification} shows in particular that (\ref{item:thm:classification:sigma}) and  (\ref{item:thm:classification:pi}) imply conditions (\ref{item:thm:classification:kappa=0,1,2,3}) and  (\ref{item:thm:classification:kappa=-infty}). 
For this reason, Theorem \ref{thm:classification} has the following consequence; we give the details of the argument in Section \ref{sec:proofs} below.  

\begin{corollary} \label{cor:fibre-bundle}
Let $f:X\to A$ be a smooth morphism from a smooth projective threefold $X$ to an abelian variety $A$.
If $\kappa(X)\geq 0$, then there is a smooth projective threefold $X'$ with the structure of an analytic fibre bundle $f':X'\to A$ over $A$, such that $X$ and $X'$ are birational over $A$.
\end{corollary}

Passage to a birational model of $X$ is necessary; the example of suitable minimal del Pezzo fibrations over elliptic curves shows that the assumption $\kappa(X)\geq 0$ is necessary as well.

By \cite[Theorem 15.1]{ueno}, the Kodaira dimension is additive for analytic fibre bundles and so, in the situation of Corollary \ref{cor:fibre-bundle}
, we have $\kappa(X)=\kappa(F)$, where $F$ denotes a fibre of $f$.
Hence, $\kappa(X)\leq \dim (F)$, which is a special case of a celebrated result of Popa and Schnell \cite{PS} (with earlier results in \cite{Za,LZ,HK}), who showed that a smooth projective variety $X$ has Kodaira dimension $\kappa(X)\leq \dim(X)-d$, if it admits a $d$-dimensional linear subspace $V\subset H^0(X,\Omega_X^1)$ of holomorphic one-forms such that any nonzero $\omega\in V$ has no zeros on $X$. 
  
The results in this paper lead us to the following two conjectures, which by the aforementioned additivity of Kodaira dimensions in analytic fibre bundles \cite[Theorem 15.1]{ueno} would generalize Popa--Schnell's result.

\begin{conjecture}\label{conj:smooth}
Let $X$ be a smooth projective variety which admits a holomorphic one-form without zeros.
Then $X$ is birational to a smooth projective variety $X'$ which admits a smooth morphism $X'\to A$ to a positive-dimensional abelian variety $A$.  
\end{conjecture}

\begin{conjecture}\label{conj:fibre-bundle}
Let $f:X\to A$ be a smooth morphism from a smooth projective variety $X$ to an abelian variety $A$.
If $\kappa(X)\geq 0$, then, up to birational equivalence, $f$ is an analytic fibre bundle.
\end{conjecture}

Conjectures \ref{conj:smooth} and \ref{conj:fibre-bundle} hold for surfaces by  \cite[Corollary 3.2]{Sch19} and for threefolds by  item (\ref{item:thm:classification:pi}) in Theorem \ref{thm:classification} and Corollary \ref{cor:fibre-bundle}.
If the fibres of $f$ are of general type, a weak form of Conjecture \ref{conj:fibre-bundle} had been proven by Popa--Schnell in arbitrary dimensions, see \cite[Corollary 3.2]{PS}.
Moreover, \cite[Corollary 3.2]{PS} easily implies Conjecture \ref{conj:fibre-bundle} if the fibres of $f$ are curves. 

\subsection{Why one-forms?}
Theorem \ref{thm:classification} yields a complete classification of all smooth complex projective threefolds with a one-form without zeros, and  Corollary \ref{cor:A=B'} shows that this is in fact a topological property.
It is natural to wonder if such a classification is possible also for forms of higher degree.
For top differential forms, this essentially amounts to classifying Calabi-Yau threefolds, and it is a famous open problem to show that such varieties come in finitely many topological types.
The remaining case is that of two-forms without zeros.
Two-forms on threefolds have previously been studied by Campana and Peternell \cite{cape00} who found infinitely many examples of smooth projective threefolds of general type (see e.g.\ \cite[Example 1.3.3]{cape00}) which carry two-forms without zeros.
This suggests that a classification is probably impossible in this case.

\subsection{A remark on the K\"ahler case}
It is conceivable that the methods of this paper allow to prove analogues of Theorems \ref{thm:A=B=C}, \ref{thm:classification} and \ref{thm:cup_omega_exact} also in the case of K\"ahler threefolds.
The main technical difficulty that one has to overcome is the fact that short exact sequences of abelian varieties always split after \'etale cover, while this is in general not true for short exact sequences of arbitrary complex tori.
As a consequence, item (\ref{item:thm:classification:kappa=0,1,2,3}) in Theorem \ref{thm:classification} does not remain true in the K\"ahler setting, but we expect that under the K\"ahler assumption, one can still prove that in item (\ref{item:thm:classification:kappa=0,1,2,3}) there is a smooth morphism as in item (\ref{item:thm:classification:pi}) which is in fact an analytic fibre bundle. 

\subsection*{Conventions and notation}
We work over the field of complex numbers.
A variety is an integral separated scheme of finite type over $\C$.
A minimal model is a projective variety $X$ with terminal $\Q$-factorial singularities such that $K_X$ is nef.

\section{Preliminaries}

\subsection{Analytic fibre bundles}
A proper morphism $f:X\to S$ of complex manifolds (or smooth complex projective varieties), is an analytic fibre bundle, if it is analytically locally isomorphic to a product of the base with a typical fibre $F$.
The isomorphism type of $f$ is determined by a cocycle in $H^1(S,\Aut(F))$.
Moreover, by a well-known result of Fischer and Grauert \cite{FG}, a proper morphism $f:X\to S$ of complex manifolds is an analytic fibre bundle if and only if it is isotrivial, i.e.\ all fibres are isomorphic to each other.


\subsection{Basic properties of condition (\ref{item:exact})}
In \cite[Theorems 1.2 and 1.5]{Sch19}, the second author proved the following two theorems, which are the starting point of our investigation.

\begin{theorem}[\cite{Sch19}] \label{thm:B=>C}
For any compact K\"ahler manifold $X$, we have (\ref{item:closed}) $\Rightarrow$ (\ref{item:exact}). 
\end{theorem}

\begin{theorem}[\cite{Sch19}]  \label{thm:C=>...}
Let $X$ be a compact K\"ahler manifold with a holomorphic one-form $\omega$ such that the complex $(H^\ast(X,\C),\wedge \omega)$ given by cup product with $\omega$ is exact.
Then the analytic space $Z(\omega)$ given by the zeros of $\omega\in H^0(X,\Omega_X^1)$ has the following properties.
\begin{enumerate}[(1)]
\item For any connected component $Z\subset Z(\omega)$ with $d=\dim Z$,
$$
H^{d}(Z,{\omega_X} |_Z)=0 . 
$$ 
In particular, $\omega$ does not have any isolated zero. \label{item:thm:Z(omega)}

\item If $f:X\to A$ is a holomorphic map to a complex torus $A$ such that $\omega\in f^\ast H^0(A,\Omega_A^1)$, then $f(X)\subset A$ is fibred by tori. 
\label{item:thm:fibred}  
\end{enumerate}
\end{theorem}

We will also use the following lemma, which is also crucial in \cite[Theorem A.1]{Sch19}.

\begin{lemma}\label{lem:contracted-divisor}
Let $X$ be a compact K\"ahler manifold and let $f:X\to A$ be a morphism to a complex torus $A$ which is generated by the image $f(X)$.
Assume that there is a one-form $\alpha\in H^0(A,\Omega_A^1)$, such that $(H^\ast(X,\C),\wedge f^\ast \alpha)$ is exact.
If there is a prime divisor $D\subset X$ with $\dim (f(D))=0$, then $A$ is an elliptic curve and $D$ is linearly equivalent to some rational multiple of a general fibre of $f$. 
\end{lemma}

\begin{proof} 
Let $a:X\to \Alb(X)$ be the Albanese morphism.
Then there is a morphism $\pi:\Alb(X)\to A$ with $f=\pi\circ a$.
Let $B \subset \Alb(X)$ the subtorus generated by $a(D)$.
Since $f(D)$ is a point, $a(D)$ is contracted by $\pi$ to a point and so $\pi$ factors through $\Alb(X)\to \Alb(X)/B$.
Hence, up to replacing $A$ by $\Alb(X)/B$, we may assume that
\begin{align} \label{eq:A=AlbX/B}
A=\Alb(X)/B.
\end{align} 
Then there is a natural short exact sequence
$$
0\longrightarrow  H^{1,0}(A)\longrightarrow  H^{1,0}(\Alb(X)) \longrightarrow H^{1,0}(B) \longrightarrow 0 .
$$ 

Since  $\dim (f(D))=0$, $[D]\wedge f^\ast \alpha=0$ and so by exactness of $\wedge f^\ast\alpha$ we get 
$$
[D]= a^\ast \overline \beta\wedge f^\ast \alpha
$$
Since $[D]$ is a real class,
$
[D]= a^\ast \beta\wedge f^\ast \overline \alpha .
$
This implies
$$
[D]\wedge a^\ast \beta=0.
$$
That is, the pullback of the holomorphic one-form $\beta$ is identically zero on $D$.
Since $B\subset \Alb(X)$ is generated by $a(D)$, it follows that $\beta$ restricts to zero on $B$ and so it lies in the image of $H^{1,0}(A)$.
That is, $\beta$ lies in the image of $f^\ast$ and so  
$$
[D]\in f^\ast H^2(A,\Q).
$$
This implies that there is a line bundle $L$ on $A$ with $f^\ast L\cong \mathcal O_X(j D)$ for some integer $j\geq 1$.
Let $Z\to f(X)$ be the normalization and let $i:Z\to A$ be the natural map.
Since $X$ is normal, $f$ induces a morphism $f':X\to Z$.
Applying the Stein factorization to $f'$, we conclude that there is a positive integer $m$ (equal to the degree of the finite map in the Stein factorization of $f'$) and a natural isomorphism
$$
H^0(X,f^\ast L)\cong H^0(Z,i^\ast L^{\otimes m}).
$$
Since $f^\ast L\cong \mathcal O_X(j D)$ is effective, we find that $i^\ast L^{\otimes m}$ admits a section $s$ whose pullback to $X$ vanishes along $D$ with multiplicity $jm$.
Since $f(D)$ is a point, $i(\{s=0\})$ is a point as well.
Since $i:Z\to A$ is finite, $\{s=0\}$ is a point and so $Z$ is a curve.
This implies that $f(X)$ must be a curve as well. 
Since $f(X)$ is a curve, its normalization is an elliptic curve by exactness of $(H^\ast(X,\C),\wedge f^\ast \alpha)$.
Since $f(X)\subset A$, it must be a smooth elliptic curve and since it
generates $A$, the latter is an elliptic curve as well.
Since $L$ is a line bundle on the elliptic curve $A$, $f^\ast L\cong \mathcal O_X(jD)$ implies that $D$ is linearly equivalent to a rational multiple of a general fibre of $f$.
This concludes the lemma.
\end{proof}

\begin{corollary}
In the situation of Lemma \ref{lem:contracted-divisor}, the Stein factorization of $f$ yields a morphism $g:X\to E$ to an elliptic curve $E$ with irreducible fibres.
\end{corollary}
\begin{proof}
Since $A$ is an elliptic curve, the Stein factorization of $f$ yields a morphism $g:X\to C$ to a smooth projective curve $C$.
By assumptions, there is a one-form $\beta\in H^0(C,\Omega_C^1)$ such that $(H^\ast(X,\C),\wedge g^\ast \beta)$ is exact.
This implies $g(C)=1$.
Applying Lemma \ref{lem:contracted-divisor} to the irreducible components of the fibres of $g$ then shows that $g$ has irreducible fibres, as we want.
\end{proof}

\begin{lemma} \label{lem:c1c2=c3=0}
Let $X$ be a compact K\"ahler threefold which satisfies (\ref{item:exact}).
Then 
$$
c_1c_2(X)=c_3(X)=0.
$$
\end{lemma}
\begin{proof}
Condition (\ref{item:exact}) implies immediately $\chi(X,\Omega_X^p)=0$ for all $p$ and so the claim follows from Riemann--Roch.
\end{proof}

\section{Reduction to minimal threefolds or Mori fibre spaces}

\begin{proposition}  \label{prop:MMP}
Let $X$ be a smooth complex projective threefold with a one-form $\omega \in H^0(X,\Omega_X^1)$ such that $(H^\ast(X,\C),\wedge \omega)$ is exact.
If $K_X$ is not nef and $X$ does not carry the structure of a Mori fibre space, then there is a smooth projective threefold $Y$ such that $X$ is the blow-up of $Y$ along a smooth elliptic curve $E$.
Moreover, if $\omega'\in H^0(Y,\Omega_Y^1)$ denotes the one-form induced by $\omega$, then $(H^\ast(Y,\C),\wedge \omega')$ is exact and $\omega'|_E$ is nonzero.
\end{proposition}
\begin{proof} 
If $K_X$ is not nef and $X$ does not carry the structure of a Mori fibre space, then by \cite[Theorem 3.3]{mori-annals}, there is a divisorial contraction 
$f: X\rightarrow Y$ whose exceptional divisor $E$
has one of the following two properties:

\begin{itemize}
\item $f|_E: E\rightarrow f(E)$ is a $\mathbb{P}^1$-bundle over a smooth curve $C=f(E)$,  $Y$ is smooth and $X=Bl_CY$;
\item $f(E)$ is a point. 
\end{itemize} 
Moreover, in both cases, $E$ contains a curve which has negative self-intersection with $E$.
It thus follows from  Lemma \ref{lem:contracted-divisor}, applied to the Albanese map of $X$, that $E$ cannot be contracted to a point and so $f$ must be the blow-up along a smooth curve $C\subset Y$.
The formula for the cohomology of blow-ups shows that exactness of $(H^\ast(X,\C),\wedge \omega)$ implies that  $(H^\ast(Y,\C),\wedge \omega')$ is exact, $C$ is an elliptic curve and $\omega'|_C$ is nonzero. 
This concludes the proposition. 
\end{proof}

\begin{corollary} \label{cor:MMP}
Let $X$ be a smooth complex projective threefold which satisfies condition (\ref{item:exact}).
Then there is a smooth projective threefold $X^{\min}$ with a birational morphism $\sigma:X\to X^{\min}$, which is given as a sequence of blow-ups along smooth elliptic curves that are not contracted via the natural map to $\Alb(X^{\min})$.
Moreover, $X^{\min}$ satisfies (\ref{item:exact}) and it is either minimal or a Mori fibre space. 
\end{corollary}
\begin{proof}
This is a direct consequence of Proposition \ref{prop:MMP},  
where we note that an elliptic curve $E$ on a smooth projective variety $X$ is contracted via the Albanese map of $X$ if and only if any holomorphic one-form on $X$ restricts trivially on $E$.
\end{proof}

By Proposition \ref{prop:MMP} and Corollary \ref{cor:MMP}, the proof of Theorems \ref{thm:classification} and \ref{thm:cup_omega_exact} reduce to the case where $X$ is either minimal, or it admits the structure of a Mori fibre space.

The following corollary of the above discussion generalizes the main result of Luo and Zhang in \cite{LZ}.

\begin{corollary} \label{cor:Xnot_gentype}
A smooth projective threefold $X$ which satisfies condition (\ref{item:holo}), (\ref{item:closed}) or (\ref{item:exact}) is not of general type.
\end{corollary}
\begin{proof}
Since (\ref{item:holo}) $\Rightarrow $ (\ref{item:closed}) is clear and (\ref{item:closed})  $\Rightarrow$ (\ref{item:exact}) by \cite{Sch19} (see Theorem \ref{thm:B=>C}), we may assume that $X$ satisfies (\ref{item:exact}).
For a contradiction, we assume that $X$ is of general type.
By Corollary \ref{cor:MMP}, we may additionally assume that $X$ is minimal.
By the Miyaoka--Yau inequality, 
$$
0>c_1^3(X)\geq \frac{8}{3} c_1c_2(X) .
$$
This contradicts Lemma \ref{lem:c1c2=c3=0}, which concludes the corollary.
\end{proof}

\section{1-forms on threefolds of non-negative Kodaira dimension}

In the case of non-negative Kodaira dimension, our main results will follow from:

\begin{theorem} \label{thm:kod=0,1,2,3}
Let $X$ be a smooth complex projective threefold of non-negative Kodaira dimension and with $K_X$ nef.
Assume that $X$ satisfies  condition (\ref{item:exact}). 

Then there is a finite \'etale covering $\tau:X'\to X$ which splits into a product $X'\cong A'\times S'$, where
$A'$ is an abelian variety of positive dimension. 
\end{theorem}

The proof of the above theorem occupies the following three sections and the final arguments will be summarized in Section \ref{sec:proof:thm:kod=0,1,2,3}. 
Before we turn to the proofs, let us note the following consequence.

\begin{corollary} \label{cor:kod=0,1,2,3}
In the notation of Theorem \ref{thm:kod=0,1,2,3}, let $\omega\in H^0(X,\Omega_X^1)$ be the one-form from condition (\ref{item:exact}).
Then the following holds:
\begin{enumerate}
\item up to passing to a finite \'etale covering of $X'$, we may assume that $A'$ is simple and $\tau^\ast \omega$ restricts to a nonzero form on $A'\times \{s\}$ for all $s\in S'$;\label{item:A'=simple}
\item 
$S'$ is smooth projective with $\kappa(S')=\kappa(X)$; 
\label{item:kappa(S)}
\item $\tau^\ast \omega$ has no zeros on $X'$ and so $\omega$ has no zeros on $X$; \label{item:omega} 
\item There is a smooth morphism $\pi:X\to A$ to an abelian variety $A$ such that for any $s'\in S'$, the composition
 \begin{align} \label{eq:A'->A}
A'\cong A'\times\{s'\}\hookrightarrow A'\times S'\cong X'\stackrel{\tau}\longrightarrow X\stackrel{\pi}\longrightarrow A
\end{align}
is finite \'etale. 
\label{item:pi}
\end{enumerate}
\end{corollary}

\begin{proof}
We have
$$
\tau^\ast\omega\in H^0(X',\Omega_{X'}^1)\cong H^0(A',\Omega_{A'}^1)\oplus H^0(S',\Omega_{S'}^1) .
$$
To see that $\tau^\ast \omega$ restricts to a nonzero form on $A'\times \{s\}$ for all $s\in S'$, it suffices to show that it does not map to zero under the projection to $ H^0(A',\Omega_{A'}^1)$.
If it does map to zero, then $S'$ satisfies the equivalent conditions (\ref{item:holo}) $\Leftrightarrow$ (\ref{item:closed}) $\Leftrightarrow$ (\ref{item:exact}).
This implies by \cite[Corollary 3.2]{Sch19} that some \'etale cover of $S'$ splits off a positive-dimensional simple abelian variety and the restriction of the pullback of $\omega$ to that factor is non-trivial, as we want.
Hence, up to passing to another \'etale cover and replacing the decomposition $X'\cong A'\times S'$ by another one, we may assume that $\tau^\ast \omega$ restricts to a nonzero form on $A'\times \{s\}$ for all $s\in S'$.
Up to passing to another finite \'etale cover of $X'$, we may by the complete reducibility theorem also assume that $A'$ is simple and so item (\ref{item:A'=simple}) of the corollary holds.

Since $\tau$ is \'etale and $X$ is smooth, so is $X'$. 
Since $X'\cong A'\times S'$, $S'$ is smooth of Kodaira dimension $\kappa(S')=\kappa(X')=\kappa(X)$, as claimed in item (\ref{item:kappa(S)}).
Moreover, item (\ref{item:omega}) is an immediate consequence of item (\ref{item:A'=simple}).

It remains to prove (\ref{item:pi}).
Consider the Albanese map $a:X\to \Alb(X)$.
Since $A'$ is simple and $\tau^\ast \omega$ restricts non-trivially to $A'\times \{s'\}$ for all $s'\in S'$, we find that for all $s'\in S'$, the image of the natural composition
$$
A'\cong A'\times\{s'\}\hookrightarrow A'\times S'\cong X'\stackrel{\tau}\longrightarrow X\stackrel{a}\longrightarrow \Alb(X)
$$
is the translate of an abelian subvariety of $\Alb(X)$ that is isogeneous to $A'$.
Since $\Alb(X)$ is projective, there is a quotient map $\Alb(X)\to A$ such that for all  $s'\in S'$, the natural composition
$$
A'\cong A'\times\{s'\}\hookrightarrow A'\times S'\cong X'\stackrel{\tau}\longrightarrow X\stackrel{a}\longrightarrow \Alb(X)\twoheadrightarrow A
$$
is finite \'etale.
We then define $\pi:X\to A$ as composition of the Albanese map of $X$ with the projection $\Alb(X)\twoheadrightarrow A$.
With this definition, the composition (\ref{eq:A'->A}) is finite \'etale for all $s'\in S$.
This implies that   $\pi\circ \tau:X'\to A$ is smooth.
Since $\tau$ is \'etale, it follows that $\pi$ is smooth as well, as we want. 
This proves (\ref{item:pi}), which concludes the corollary.
\end{proof}

\section{Proof of Theorem \texorpdfstring{\ref{thm:kod=0,1,2,3}}{4.1}  for \texorpdfstring{$\kappa(X)=2$}{kappa=2}}
\label{sec:kappa=2}

In this section we aim to prove Theorem \ref{thm:kod=0,1,2,3} in the case where $\kappa(X)=2$.

\subsection{Preliminaries on elliptic threefolds} 

\begin{definition}
An elliptic threefold is a normal projective threefold $X$ with a morphism $f:X\to S$ to a normal projective surface whose general fibre is an elliptic curve.
We say that $f$ has trivial (or no) monodromy, if $R^1f_\ast \Q$ restricts to a trivial local system over some non-empty (Zariski) open subset $U\subset S$. 
\end{definition}

\begin{lemma} \label{lem:no-monodromy=>1-form}
Let $f:X\to S$ be an elliptic threefold with trivial monodromy and such that $X$ has rational singularities (e.g.\ terminal singularities).
Then the general fibre of $f$ is not contracted via the Albanese morphism $a:X\to \Alb(X)$.
\end{lemma}
\begin{proof}
Since $X$ has rational singularities, the Albanese morphism of any resolution $\widetilde X$ of $X$ factors through $X$, and so $a$ is defined.
The lemma then follows from Deligne's global invariant cycle theorem (see e.g.\ \cite[Theorem 4.24]{VoiII}) applied to $\widetilde X$, which implies that $\widetilde X$ carries a holomorphic one-form which restricts nontrivially on the general fibre of the natural map $\widetilde X\to S$. 
\end{proof}

We will need the following result, c.f.\  \cite[Theorem 2.7]{gra94}.

\begin{proposition} \label{prop:grassi}
Let $X$ be terminal threefold with $K_X$ nef of Kodaira dimension two and with Iitaka fibration $f:X\to S$.
If $c_1c_2(X)=0$, then  $f$ has only finitely many  singular fibers which are not multiples of a smooth elliptic curve.
Moreover, the $j$-invariants of the smooth fibres of $f$ are constant.
\end{proposition}
\begin{proof}
The Iitaka fibration $f:X\to S$ is a morphism by the abundance conjecture \cite{Ka92}.
In particular, there is a very ample divisor $A$ on $S$ such that $K_X=\lambda f^\ast A$ for some positive rational number $\lambda$. 
Let $C\subset S$ be a general element of $|A|$ and let $Y:=f^{-1}(C)$.
Since $S$ is normal and $X$ is terminal, it follows from Bertini's theorem that $Y$ and $C$ are smooth.
Moreover, $f|_Y:Y\to C$ is a minimal elliptic surface of Kodaira dimension one, because $X$ is minimal of Kodaira dimension two and Iitaka fibration $f$.
Since $Y$ is smooth and contained in the smooth locus of $X$, we have a short exact sequence of vector bundles on $Y$:
$$
0\longrightarrow T_Y\longrightarrow T_X|_Y\longrightarrow f^\ast \mathcal O_S(A) |_Y\longrightarrow 0 .
$$ 
Applying the Whitney sum formula, we deduce that the second Chern number of $Y$ is given by
\begin{align*}
c_2(Y)&=c_2(X)|_Y-c_1(Y)f^\ast A|_Y =(-\lambda)^{-1}c_1c_2(X)-c_1(Y) f^\ast A^2 .
\end{align*}
By adjunction, $c_1(Y)=(c_1(X)- f^\ast A)|_Y$, and so  
$
c_1(Y) f^\ast A^2=0
$,
as it is a multiple of $f^\ast A^{3}=0$.
Since $c_1c_2(X)=0$, we conclude $c_2(Y)=0$ from the above formula.
Hence, $Y$ is a minimal surface of Kodaira dimension one with $c_2(Y)=0$.
Since $c_2(Y)$ coincides with the sum of the Euler numbers of the singular fibres of $Y\to C$, we find by Kodaira's classification of singular fibres (see \cite{barth-etal}) that any singular fibre of $Y\to C$ is a multiple of a smooth elliptic curve.
This implies that the j-invariant $j:C\dashrightarrow \CP^1$ is not dominant and so it must be constant. 
This proves the proposition.
\end{proof}

We have the following important structure theorem of Nakayama.
To state it, recall that an elliptic threefold $f:X\to S$ is a standard elliptic fibration if $X$ is $\Q$-factorial and terminal, $f$ is equi-dimensional  and $K_X\sim_{\Q}f^\ast(K_S+\Delta)$ for an effective $\Q$-divisor $\Delta$ such that $(S,\Delta)$ is log terminal.

\begin{theorem}[{\cite[Theorem A.1]{Nak-local}}] \label{thm:nakayama}
Let $f:X\to S$ be an elliptic threefold. 
Then there is a proper birational morphism $S'\to S$ and a standard elliptic fibration $f':X'\to S'$  that is birational to $f$ over $S$, such that $K_{X'}$ is semi-ample over $S$.
\end{theorem}

\begin{lemma} \label{lem:pi_1(X)}
Let $f: X\rightarrow S$ be an elliptic threefold, such that $X$ is terminal, $\Q$-factorial and $K_X$ is $f$-nef.
Let $S'\to S$ and $f':X'\to S'$ be as in Theorem \ref{thm:nakayama}.
Then there is a smooth open subset $U'\subset S'$, whose complement in $S'$ is zero-dimensional, and such that the base change $X'_{U'}:={f'}^{-1}(U')$ is a smooth threefold.
Moreover, for any such open subset $U'\subset S'$, the natural birational map $X'_{U'}\dashrightarrow X$ induces an isomorphism
$$
\pi_1(X'_{U'})\cong \pi_1(X^{\sm}) ,
$$ 
where $X^{\sm}\subset X$ denotes the smooth locus of $X$.
\end{lemma}
\begin{proof}
Since $f'$ is a standard elliptic fibration, $X'$ has only terminal $\Q$-factorial singularities, $f'$ is equi-dimensional and $S'$ is normal.
In particular, $X'$ and $S'$ have isolated singularities.
Hence, there is an open subset $U'\subset S'$ whose complement in $S'$ is zero-dimensional, and such that the base change $X'_{U'}:={f'}^{-1}(U')$ is smooth.

Let now $U'\subset S'$ be any such subset.
Since $K_{X'}$ is semi-ample over $S$, it is in particular nef over $S$.
Hence, $X$ and $X'$ are birational minimal models over $S$ and so they are  isomorphic in codimension one, see e.g.\ \cite[Theorem 3.52(2)]{kollar-mori}.  
Since $S'\setminus U'$ is zero-dimensional and $f'$ is equi-dimensional, $X'\setminus X'_{U'}$ is at most one-dimensional.
Since $X'$ and $X$ are isomorphic in codimension one, we conclude the same for $X'_{U'}$ and $X$, and hence also for $X'_{U'}$ and $X^{\sm}$, because $X$ is terminal and so it has isolated singularities.
Since $X'_{U'}$ and $X^{\sm}$ are smooth, this implies 
$
\pi_1(X'_{U'})\cong \pi_1(X^{\sm}) 
$, as we want.
\end{proof}

\subsection{Condition (C) implies trivial monodromy after \'etale cover}
\label{subsec:ell3fold:nomonodromy}

\begin{lemma}\label{lem:nontriv-monodromy}
Let $X$ be a smooth projective threefold with $K_X$ nef of Kodaira dimension two and with Iitaka fibration $f:X\to S$.
Assume that (\ref{item:exact}) holds for $X$.
If $f$ has non-trivial monodromy, then $f$ is equi-dimensional.
\end{lemma}
\begin{proof}
Assume that $f$ has non-trivial monodromy.
Since any variation of Hodge structure of weight one and rank two with a nonzero section is trivial, this implies that $R^1f_\ast \Q$ has no generically non-zero section.
Hence, the general fibre of $f$ is contracted via the Albanese map $a:X\to \Alb(X)$.
This implies that $a$ factors rationally through $f$.
Since $S$ is the base of the Iitaka fibration, it has at most klt singularities, hence rational singularities, and so any rational map from $S$ to $\Alb(X)$ is a morphism.
That is, $a$ factors through $f$.
Hence, any prime divisor $D\subset X$ which maps to a point on $S$ is contracted to a point by the Albanese map of $X$.
Lemma \ref{lem:contracted-divisor} then shows that $\Alb(X)$ is an elliptic curve and $D$ is numerically equivalent to a rational multiple of a fibre of $a$.
Since $D$ is contracted by $f$, $f$ also contracts a general fibre of $a$. 
Hence $f$ factors through $a$, which is impossible because $S$ is a surface.
This proves the lemma.
\end{proof}

\begin{remark}
The condition on the monodromy is necessary in Lemma \ref{lem:nontriv-monodromy}.
To see this, let $S$ be a canonical surface with ample canonical bundle and a single node as singularity, and with minimal resolution $\tilde S\to S$.
Then for any elliptic curve $E$, the product $X:=\tilde S\times E$ is a minimal threefold of Kodaira dimension two which satisfies (\ref{item:exact}), but the Iitaka fibration of $X$ is given by the natural map $X\to S$, which is not equi-dimensional. 
\end{remark}

We are now able to prove the following, which is the main result of Section \ref{subsec:ell3fold:nomonodromy}.

\begin{proposition}\label{prop:monodromy}
Let $f: X\rightarrow S$ be an elliptic threefold, with $X$ smooth and $K_X$ nef.
Assume that $X$ satisfies (\ref{item:exact}).
Then there is a finite \'etale  cover $\tau:\widetilde X\to X$, such that the elliptic fibration $\tilde f:\widetilde X \to \widetilde S$  that is induced by $f$ via Stein factorization has trivial monodromy. 
\end{proposition}

\begin{proof}
By Lemma \ref{lem:nontriv-monodromy}, we may assume that $f$ is equi-dimensional.
By Proposition \ref{prop:grassi}, $f$ is generically isotrivial and only finitely many singular fibers of $f$ are not multiples of a smooth elliptic curve. 
This implies that there is an open subset $U\subset S$ whose complement is zero-dimensional, such that $R^1f_\ast \Q|_U$ is a local system  (which above the multiple fibres can be checked via topological base change).
Moreover, this local system has finite monodromy, because the identity component $\Aut^0(F)$ of $\Aut(F)$ acts trivially on $H^1(F,\Q)$.

It follows that there is a finite \'etale cover $\tilde U\to U$ such that the base change $X\times_S\tilde U$ is an elliptic threefold over $\tilde U$ with trivial monodromy.
Note that $X\times_S\tilde U$ is a finite \'etale cover of $f^{-1}(U)$.
Since $f$ is equi-dimensional, $\pi_1(f^{-1}(U))=\pi_1(X)$, and so this finite \'etale cover extends to a finite \'etale cover $\tau:\tilde X\to X$, such that the map $\tilde f:\tilde X\to \tilde S$, induced via Stein factorization of $f\circ \tau:\tilde X\to S$, is an elliptic threefold with trivial monodromy.
This finishes the proof of the proposition.
\end{proof}

\begin{remark} 
The case where $X$ is a product of a curve with a bi-elliptic surfaces shows that the \'etale covering performed in Proposition \ref{prop:monodromy} is really necessary.
\end{remark}

\subsection{Classification of minimal elliptic threefolds with trivial monodromy}

Proposition \ref{prop:monodromy} reduces the proof of Theorem \ref{thm:kod=0,1,2,3} for $\kappa(X)=2$ to the case of minimal elliptic threefolds $X\to S$ with trivial monodromy.
Even though $X$ is smooth in that situation, it is not much harder to classify more generally such threefolds with terminal singularities.
This is the content of the following theorem.
To state the result, recall that a finite morphism $f:X'\to X$ between normal varieties is called quasi-\'etale if it is \'etale in codimension one, see e.g.\ \cite{GKP16}; if $X$ is smooth, then this implies that $f$ is \'etale.
In particular, $f$ is ramified at most at the singular points of $X$.

\begin{theorem} \label{thm:elliptic-threefolds:basic} 
Let $f:X\to S$ be an elliptic threefold with trivial monodromy, where $X$ is terminal, $\Q$-factorial and $K_X$ is nef. 
Then there is a finite quasi-\'etale covering $\tau:X''\to X$ with $X''\cong S''\times E$, where $E$ is an elliptic curve and $S''$ is a smooth projective surface with a generically finite map to $S$.
\end{theorem}

In the proof of the above theorem, we will use the following local result.

\begin{lemma} \label{lem:multiple-fibres}
Let $f:X\to \Delta$ be a proper morphism of complex manifolds over the disc $\Delta$, which is a submersion over the punctured disc $\Delta^*:=\Delta \setminus \{0\}$.
Assume that the special fibre $X_0$ of $f$ is irreducible and of multiplicity $m$.
Assume that there is a morphism $g:X\to F$ to a compact complex manifold $F$ which restricts to a finite morphism on a general fibre of $f$.
Let $S:=g^{-1}(x)$ be a general fibre of $g$.
Then the normalization $X'$ of the base change $X\times_\Delta S$ is smooth, the family $X'\to S$ has reduced fibres and the natural map $X'\to X$ is \'etale.
If furthermore the reduced fibre $X_0^{\red}$ is smooth, then $X'\to S$ is smooth.
\end{lemma}

\begin{proof}
Since $f$ is a submersion over the punctured disc, it suffices to prove the lemma after shrinking $\Delta$, if necessary.
By Sard's theorem, $S$ is smooth.
Moreover, the natural map $S\to \Delta$ is submersive away from the origin of $\Delta$.
Let $S'$ be a connected component of $S$.
Since $S'$ is connected, up to shrinking $\Delta$, $S'$ meets the central fibre $X_0$ in a single point and so $S'\to \Delta$ is a cyclic cover and we denote its degree by $k$.
Since $X_0$ has multiplicity $m$, $m$ divides $k$.
Conversely, the morphism $X_0^{\red}\to F$ induced by $g$ is generically smooth and so the preimage of $x\in F$ in $X_0^{\red}$ is given by disjoint reduced points.
Since the intersection of $S'$ with $X_0$ is a single point, we find that this intersection has multiplicity at most $m$.
Hence, $k=m$ and $S'\to \Delta$ is a cyclic cover of degree by $m$.
It then follows from a well-known local computation, see e.g.\ \cite[Proposition III.9.1]{barth-etal}, that the normalization $X'$ of the base change $X\times_\Delta S$ is smooth, the family $X'\to S$ has reduced fibres and the natural map $X'\to X$ is \'etale.
Moreover, the central fibre of $X'\to S$ is an \'etale cover of $X_0^{\red}$.
Hence, $X'\to S$  is smooth if $X_0^{\red}$ is smooth.
This concludes the lemma.
\end{proof}

\begin{proof}[Proof of Theorem \ref{thm:elliptic-threefolds:basic}]
By Theorem \ref{thm:nakayama}, there is a birational morphism $S'\to S$ and a standard elliptic fibration $f':X'\to S'$  that is birational to $f$ over $S$ and such that $K_{X'}$ is nef over $S$.
By Lemma \ref{lem:pi_1(X)}, there is a smooth open subset $U'\subset S'$ whose complement is zero-dimensional and such that $X'_{U'}:={f'}^{-1}(U')$ is smooth and the birational map $X'_{U'}\dashrightarrow X$ induces an isomorphism 
\begin{align} \label{eq:pi_1(X)}
\pi_1(X'_{U'})\cong \pi_1(X^{\sm}) .
\end{align} 

Since $X'$ has only terminal singularities, there is a well-defined Albanese map $a':X'\to \Alb(X')$, obtained by observing that the Albanese map of any desingularization of $X'$ factors through $X'$.
By Lemma \ref{lem:no-monodromy=>1-form}, the general fibre of $f':X'\to S'$ is not contracted by $a'$.  
Hence, the general fibre of $f'$ is via $a'$ mapped to a translate of a fixed elliptic curve $E\subset \Alb(X')$.
Since $X'$ is projective, we can dualize this inclusion to get a surjection $\Alb(X')\twoheadrightarrow E$.  
Composing this with $a'$, we get a surjection 
$$
g:X'\longrightarrow E ,
$$
which restricts to finite \'etale covers on general fibres of $f'$.
Taking the Stein factorization, we may assume that $g$ has connected fibres.
(Note that the target of the Stein factorization will receive a surjection from the general fibres of $f'$, which are elliptic curves, and so it cannot be a curve of genus $\geq 2$.)

Since $X'$ is terminal, it has isolated singularities.
A general fibre $\tilde S=g^{-1}(e)$ of $g$ is thus smooth by Bertini's theorem and we consider the normalization $\tilde X$ of the base change $X'\times_{S'}\tilde S$.
We then get a commutative diagram
$$
\xymatrix{ 
\tilde X\ar[r]\ar[d]_{\tilde f} &X'\ar[d]^{f'}\\
\tilde S \ar[r] &S' .} 
$$
Let $\tilde U\subset \tilde S$ be the preimage of $U'\subset S'$.
We consider the base change $\tilde X_{\tilde U}=\tilde f^{-1}(\tilde U)\subset \tilde X$.
Since $f$ has trivial monodromy by assumptions, the same holds for $f'$.
Since additionally $K_{X'}$ is nef over $S'$, the base change $X'_Z$ to a general  hyperplane section $Z\subset S'$ is a smooth minimal elliptic surface with trivial monodromy.
This implies that the singular fibres are multiples of smooth elliptic curves: they are (multiples of) one of the fibres in Kodaira's table \cite[p.\ 201]{barth-etal} and additionally have second Betti number at least $2$, because of the triviality of the monodromy and topological proper base change.
Hence, away from finitely many points in $S'$, all singular fibres of $f'$ are multiples of smooth elliptic curves. 

Let $C'\subset U'$ be a general hyperplane section and let $\tilde C\subset \tilde U$ be its preimage in $\tilde U$.
Applying Lemma \ref{lem:multiple-fibres} to the base change of $\tilde X_{\tilde U}$ and $X'_{U'}$ to $\tilde C$ and $C'$, respectively, we find the following: up to removing finitely many points from $U'$ and $\tilde U$, we may assume that $\tilde X_{\tilde U}\to \tilde U$ is a smooth elliptic fibre bundle and $\tilde X_{\tilde U}\to X'_{U'}$ is \'etale.
Since this bundle has a section by construction, the existence of a fine moduli space for elliptic curves with level structure shows that $\tilde X_{\tilde U}\cong \tilde U\times E$ for an elliptic curve $E$.
 
By \cite[Theorem 3.8]{GKP16}, any finite \'etale cover of $X^{\sm}$ extends to a finite quasi-\'etale cover of $X$.
Since
$
\pi_1(X'_{U'})\cong \pi_1(X^{\sm})
$, the finite \'etale cover $\tilde X_{\tilde U}\to X_{U'}'$ is thus birational to a finite quasi-\'etale covering
$$
X''\to X
$$
of $X$.
Since $\tilde X_{\tilde U}\cong \tilde U\times E$, we conclude that $X''$ is birational to $S''\times E$, where $S''$ is a minimal surface and $E$ is an elliptic curve.
Since $K_X$ is nef, so is $K_{X''}$.
Moreover, $X''$ is terminal by \cite[Proposition 5.20]{kollar-mori}, because it is a finite quasi-\'etale cover of a terminal threefold.

By \cite[Theorem 6.25]{kollar-mori}, there is a $\Q$-factorialization $\sigma:\tilde X''\to X'' $, i.e.\ a proper birational morphism which is an isomorphism in codimension one such that $\tilde X''$ is $\Q$-factorial, terminal and $K_{\tilde X''}$ is nef.
Hence, $\tilde X''$ and $S''\times E$ are birational minimal models and so they are connected by a sequence of flops (see \cite{kol89}). 
Since  $S''\times E$ does not admit any non-trivial flop,
$
\tilde X''\cong S''\times E .
$ 
On the other hand, the product $S''\times E$ does not admit a small contraction to a terminal threefold (because any rational curve on it maps to a point on the second factor and so it sweeps out a divisor on $S''\times E$).
Hence, $X''\cong \tilde X''$ and so
$
X''\cong S''\times E
$,
as we want.
This concludes the proof. 
\end{proof}


\subsection{Proof of Theorem \ref{thm:kod=0,1,2,3} for \texorpdfstring{$\kappa(X)=2$}{kappa=2}}
\label{subsec:kappa=2}

Since $X$ is a smooth projective threefold with $K_X$ nef, the Iitaka fibration $f:X\to S$ is a morphism by the abundance conjecture for threefolds, see \cite{Ka92}.
This endows $X$ with the structure of an elliptic threefold.
By  Proposition \ref{prop:monodromy}, 
there exists a finite \'etale covering $X'\to X$, such that $f$ induces an elliptic fibration $f':X'\to S'$ without monodromy. 
Hence, by Theorem \ref{thm:elliptic-threefolds:basic}, there is a finite \'etale cover $X''\to X'$, such that $X''\cong S''\times E$ for some smooth projective surface $S''$ and an elliptic curve $E$.  
This concludes Theorem \ref{thm:kod=0,1,2,3} if $\kappa(X)=2$. 


\section{Proof of Theorem \texorpdfstring{\ref{thm:kod=0,1,2,3}}{4.1}  for \texorpdfstring{$\kappa(X)=1$}{kappa=1}} \label{sec:kappa=1}

In this section we aim to prove Theorem \ref{thm:kod=0,1,2,3} in the case where $X$ is a minimal smooth projective threefold with $\kappa(X)=1$.
By the abundance conjecture (which is known in dimension three, see \cite{Ka92}), the Iitaka fibration of $K_X$ yields a morphism $f:X\to C$ to a smooth projective curve $C$.
We may thus consider the diagram 
$$
\xymatrix{
   X\ar[d]_-{f} \ar[r]^-a &\Alb(X) \\
  C,}
$$
where $a$ denotes the Albanese morphism.

After collecting some preliminary results in Section \ref{subsec:preliminaries}, we  treat in Sections  \ref{subsec:X_c=curve,phi(F)=surface}, \ref{subsec:X_c=curve,phi(F)=curve} and \ref{subsec:X_c=curve,phi(F)=pt} below the cases where the Albanese image $a(F)\subset \Alb(X)$ of a general fibre $F$ of $f$ has dimension two, one and zero, respectively. 

\subsection{Preliminaries} \label{subsec:preliminaries}

We begin with the following simple observation, cf.\ \cite[Theorem 1.5]{gra94}.

\begin{lemma} \label{lem:Grassi:c1c2=0}
Let $X$ be a smooth projective threefold with Iitaka fibration $f:X\to C$ to a smooth projective curve $C$.
Then, $c_1c_2(X)=0$ if and only if the smooth fibres of $f$ are bi-elliptic or abelian surfaces.
\end{lemma}

\begin{proof}
Some multiple $mK_X$ is linearly equivalent to the pullback of an ample divisor on $C$.
Hence, $mK_X$ is numerically equivalent to a positive multiple of a general fibre $F$ of $f$.
Hence, $c_1c_2(X)=0$ if and only if $c_2(X)$ restricts to zero on $F$, which is to say that $c_2(F)=0$ because $F$ has trivial normal bundle.
Since $F$ is a minimal surface of Kodaira dimension zero, $c_2(F)=0$ implies that $F$ is either bi-elliptic or an abelian surface, see e.g.\ \cite[Chapter VI.1]{barth-etal}. 
This concludes the lemma.
\end{proof}

Next, we recall the following definition, see \cite[Definition 1.9]{cape00}.

\begin{definition}
Let $X$ be a smooth projective threefold with a surjective morphism $f: X\rightarrow C$ to a smooth curve $C$. 
A holomorphic two-form $\xi$ on $X$ is vertical (with respect to $f$) 
if the annihilator of $\xi$ on the tangent space $T_{X,x}$ of $X$ at a general point $x\in X$ is tangential to the fibre of $f$ at $x$. 
Equivalently, $\xi$ is vertical if and only if it has trivial image via the natural map $\Omega_X^2\to\Omega^2_{X/C}$. 
\end{definition}

With this terminology, Campana and Peternell proved the following: 

\begin{theorem}[{\cite[Theorem 4.2]{cape00}}]\label{thm:CP}
Let $X$ be a minimal smooth projective threefold with $\kappa(X)=1$. 
Let $f: X\rightarrow C$ be the Iitaka fibration.
Assume that there is a holomorphic two-form $\eta$ on $X$ which is not vertical with respect to $f$.
Then there is a finite morphism $C'\to C$ such that the normalization $X'$ of the base change $X\times_CC'$ splits into a product $X'\cong F\times C'$.
In particular, $f$ is quasi-smooth, i.e.\ all singular fibres of $f$ are multiple fibres.  
\end{theorem}

The multiplicities of the singular fibres define an effective divisor $D$ on $C$ and this divisor defines an orbifold structure on $C$, see e.g.\ \cite[Section 2.1.3]{FM}.
We say that such an orbifold is good, if there is an orbifold \'etale covering with trivial orbifold structure.

\begin{proposition} \label{prop:good-cover}
Let $X$ be a smooth projective threefold with a morphism $f:X\to C$ to a smooth projective curve $C$.
Assume that $f$ is quasi-smooth with typical fibre an abelian surface $F$, i.e.\ the smooth fibres of $f$ are isomorphic to a fixed abelian surface $F$ and the singular fibres are multiples of an abelian surface isogeneous to $F$. 
If the orbifold structure on $C$ that is induced by the multiple fibres of $f$ is good, then there is a finite \'etale covering $\tau:X'\to X$ such that $X'\cong F\times C'$, where $C'$ is a finite cover of $C$. 
\end{proposition}

\begin{proof}
If the given orbifold structure is good, then there is a branched cover $C'\to C$, branched exactly at the orbifold points of $C$ with the prescribed multiplicities.
The base change $X\times_CC'$ is singular along the preimage of the singular fibres of $f$.
Let $X'$ be the normalization of $X\times_CC'$.
Then a local analysis shows that $X'$ is smooth, the natural map $\tau:X'\to X$ is \'etale and $X'\to C'$ is smooth. 
Since $X'$ is projective, the same argument as in the proof of \cite[Theorem 4.2]{cape00} shows that up to a further \'etale  base change $X'\cong F\times C'$, as we want.\footnote{Note that in \textit{loc. cit.} it is claimed that this splitting holds even in the K\"ahler setting, but this seems to be incorrect, because even after \'etale cover, the extension (\ref{ses:AlbX'}) does in the non-polarized setting in general not split.}
We repeat the argument for convenience of the reader.
Consider the natural map $f':X'\to C'$ and note that $R^1f'_\ast \Q$ is a local system.
Since $\Aut(F)$ acts on $H^1(F,\Q)$ via a finite quotient, $R^1f'_\ast \Q$ has finite monodromy.
Hence, after a suitable \'etale base change, we may assume that $R^1f'_\ast \Q$ is trivial and so $b_1(X')=b_1(C')+b_1(F)$.
We thus get a short exact sequence
\begin{align} \label{ses:AlbX'}
0\to F\to \Alb(X')\to \Jac(C')\to 0.
\end{align}
Since $X'$ is projective, this sequence splits after a suitable \'etale cover of $X'$ and so $X'\cong F\times C'$, as we want.
This concludes the proposition.
\end{proof}

We finally recall the classification of all good orbifolds, see e.g.\ \cite[Corollary 2.29]{CHK00}.

\begin{theorem} \label{thm:orbifold-cover}
Let $C$ be a smooth projective curve with an effective divisor $D\in \Div(C)$.
The orbifold $(C,D)$ is good unless $C\cong \CP^1$ and one of the following holds:
\begin{enumerate}[(1)]
\item $D$ consists of one point with some multiplicity;
\item $D$ consists of two points with different multiplicities. 
\end{enumerate}
\end{theorem}

\subsection{The general fibre of $f$ is not contracted via the Albanese map}
\label{subsec:X_c=curve,phi(F)=surface}

By Lemmas \ref{lem:c1c2=c3=0} and \ref{lem:Grassi:c1c2=0}, a general fiber $F$ of $f$ is an abelian or a bi-elliptic surface. 
Since $\dim(a(F))=2$ in the present case, $F$ must be an abelian surface.
In particular, $a(F)\subset \Alb(X)$ is the translate of a fixed abelian subvariety of $\Alb(X)$.
This implies that the pullback of a general holomorphic two-form from $\Alb(X)$ to $X$ is not vertical with respect to $f$.
Hence, Theorem \ref{thm:CP} applies and we see that $f$ is quasi-smooth.
The multiplicities of the singular fibres define an effective divisor $D$ on $C$ and this divisor defines an orbifold structure on $C$, see e.g.\ \cite[Section 2.1.3]{FM}.
We claim that this orbifold structure is good, i.e.\ there is an orbifold \'etale covering $C'\to C$ which has trivial orbifold structure.  
To this end, we may by Theorem \ref{thm:orbifold-cover} assume  $g(C)=0$.
Then the Leray spectral sequence yields $H^1(X,\mathcal O_X)\cong H^0(C,R^1f_\ast \mathcal O_X)$.
This implies $b_1(X)=4$, because a holomorphic one-form (hence also an anti-holomorphic one-form) on $X$ which vanishes on one smooth fibre of $f$ vanishes at any smooth fibre, since the image of these fibres in $\Alb(X)$ are translates of the same abelian subvariety.
Hence, $a:X\to \Alb(X)$ induces an isogeny on the general fibre of $f$.
For a general point $p\in \Alb(X)$, the preimage $C':=a^{-1}(p)$ is a smooth projective curve and the natural map $C'\to C$ is branched exactly at the given points of the orbifold structure on $C$.
Hence, the orbifold structure is good, as we want.
The proof of Theorem \ref{thm:kod=0,1,2,3} in the case treated in Section \ref{subsec:X_c=curve,phi(F)=surface} thus follows from Proposition \ref{prop:good-cover}.

\subsection{The general fibre of $f$ is via the Albanese map contracted to a curve} \label{subsec:X_c=curve,phi(F)=curve}

Recall that the smooth fibres of $f$ are either bi-elliptic or abelian surfaces by Lemmas \ref{lem:c1c2=c3=0} and \ref{lem:Grassi:c1c2=0}.
In the present case, this implies that the fibres of $f$ are mapped via the Albanese morphism $a:X\to \Alb(X)$ to translates of a fixed elliptic curve $E$ in $\Alb(X)$.
Since $X$ is projective, we can dualize the inclusion $E\subset \Alb(X)$ and get a surjection $\Alb(X)\to E$.
This gives rise to a morphism 
$$
g:X\longrightarrow E,
$$ 
which when restricted to the fibres of $f$ coincides with $a$ up to isogeny.
Taking the Stein factorization, we may assume that $g$ has connected fibres.
(Note that the target of the Stein factorization receives a surjection from the fibres of $f$, which are bi-elliptic or abelian surfaces, and so it cannot be a curve of genus $\geq 2$.)

We proceed in several steps.

\medskip
\textbf{Step 1.}
Up to replacing $X$ by a finite  \'etale covering that is induced by a finite \'etale cover of $E$, we may assume that there is a finite \'etale cover $\tilde E\to E$ such that for general $c\in C$,
\begin{align} \label{eq:X'_c:2}
X_c\cong \tilde E\times F_c
\end{align}
for an elliptic curve $F_c$ which might depend on $c\in C$ and such that $g|_{X_c}:X_c\to E$ corresponds to the natural composition $\tilde E\times F_c\to\tilde E\to E$.

\begin{proof}
Consider the restriction $g|_{X_c}:X_c\to E$ of $g$ to a smooth fibre $X_c=f^{-1}(c)$ of $f$.
Even though $g$ has connected fibres, this might a priori not have connected fibres, but taking the Stein factorization, we see that there is a finite \'etale cover $\tilde E_c\to E$, such that $g|_{X_c}$ factors through a morphism $\tilde g_c:X_c\to \tilde E_c$ with connected fibres.
Since $X_c$ is bi-elliptic or abelian, there is a finite \'etale cover $\tilde E'_c\to \tilde E_c$, such that the induced \'etale cover $X_c'\to X_c$ splits into a product 
\begin{align} \label{eq:X'_c}
X'_c=\tilde E'_c\times F_c,
\end{align}
where $F_c$ is an elliptic curve which might depend on $c\in C$.

Note that $\tilde E'_c\to E$ is an \'etale cover and so it is determined by a finite index subgroup of $\pi_1(E)=\Z^{2}$.
In particular, there is a finite \'etale cover $\tilde E'\to E$ which is isomorphic to $\tilde E'_c\to E$ for all general points $c\in C$.
Similarly, there is a finite \'etale covering $\tilde E\to E$ which is isomorphic to  $\tilde E_c\to E$ for general $c\in C$.

\begin{lemma}
Up to replacing $\tilde E'$ by a further finite \'etale cover, there is a finite \'etale cover $E'\to E$, such that  $\tilde E\times_EE' \to \tilde E$ is isomorphic to $\tilde E'\to \tilde E$.
That is, there is a Cartesian diagram
$$
\xymatrix{
\tilde E'\ar[r]\ar[d] &  E'\ar[d]\\
\tilde E\ar[r]& E .
}
$$
\end{lemma}
\begin{proof}
Consider the sequence of fundamental groups $\pi_1(\tilde E')\to \pi_1(\tilde E)\to \pi_1(E)$.
Identifying all three groups with $\Z^2$, this sequence is of the form
$$
\Z^2\stackrel{A}\longrightarrow \Z^2\stackrel{B}\longrightarrow \Z^2 ,
$$
where $A,B\in \GL_2(\Q)$ are invertible matrices with integer entries.
Since $E'$ is determined by a finite index subgroup of $\pi_1(E)$, the lemma is then equivalent to finding an invertible matrix $C\in \GL_2(\Q)$ with integer coefficients, such that 
$$
BA=CB;
$$ 
the \'etale cover $E'\to E$ is then induced by the subgroup of $\pi_1(E)\cong \Z^2$ given by the image of $C:\Z^2\to \Z^2$.
The above condition is equivalent to asking that $C=BAB^{-1}$ has integer entries.
But this can easily be ensured by multiplying $A$ by a sufficiently divisible integer, which corresponds to replacing $\tilde E'$ by a further finite \'etale cover.
This proves the lemma.
\end{proof}

Let us now consider the finite \'etale cover $X':=X\times_EE'$ of $X$ together with the natural maps $g':X'\to E'$ and $f':X'\to C$.
The fibre of $f'$ above a general point $c\in C$ is given by the fibre product 
$$
X_c\times_EE'=(X_c\times_{\tilde E_c}\tilde E_c)\times_EE' \cong X_c\times_{\tilde E_c}(\tilde E\times_{E}E') \cong X_c\times _{\tilde E_c}\tilde E_c' ,
$$
which coincides with $X'_c$ from above and so it splits as in (\ref{eq:X'_c}) into a product $\tilde E'\times F_c$ of elliptic curves. 
Hence, up to replacing $X$ by $X'$, we may assume that for general $c\in C$, 
$
X_c\cong \tilde E\times F_c
$
for a finite  \'etale cover $\tilde E\to E$ and for an elliptic curve $F_c$ which might depend on $c\in C$.
That is, (\ref{eq:X'_c:2}) holds, which concludes step 1.
\end{proof}

Note that $\tilde E\to E$ is an isomorphism if and only if $f\times g:X\to C\times E$ has connected fibres.
We show next that, up to a suitable base change, we may assume that this holds true.

\medskip

\textbf{Step 2.}
In the above notation, up to replacing $X$ by a finite  \'etale cover, given by the normalization of $X\times_CC'$ for some finite map $C'\to C$, we may assume that $f\times g:X\to C\times E$ has connected fibres, i.e.\ $\tilde E\to E$ in step 1 is an isomorphism.
\begin{proof}
To prove the claim, let $X\stackrel{h}\longrightarrow S\longrightarrow C\times E$ be the Stein factorization of $f\times g$.
We claim that $S$ is smooth and the natural map $S\to C$ is a minimal elliptic surface all of whose singular fibres are multiples of an elliptic curve.

To see this, note that $S$ is normal.
By (\ref{eq:X'_c:2}), the general fibre $S_c$ of $S\to C$ is isomorphic to
$$
S_c\cong \tilde E.
$$
That is, $S\to C$ is a normal elliptic surface whose general fibres are isomorphic to $\tilde E$. 
Since $X\to C$ has connected fibres, the fibres of the natural map $S\to C$ are connected and their reductions are finite covers of $E$.
Since the arithmetic genus is constant in flat families, all fibres of $S\to C$ must be irreducible.
Let $\tilde S\to S$ be a minimal resolution. 
Since $S\to C$ has connected fibres, the induced map $\tilde S\to C$ has connected fibres as well.
Since $\tilde S\to S$ is the minimal resolution and no fibre of $S\to C$ has a rational component, it follows that $\tilde S\to C$ is relatively minimal elliptic surface.
Since the reductions of the fibres of $S\to C$ are finite (\'etale) covers of $E$, we find that all singular fibres of $\tilde S\to C$ are multiple fibres (by Kodaira's table of singular fibres of relatively minimal elliptic fibrations, see e.g.\ \cite{barth-etal}).
But this implies that $\tilde S\to S$ is an isomorphism, as we want.

%
Since $S\to C$ is a minimal elliptic surface whose general fibres are isomorphic to $\tilde E$ and whose singular fibres are multiples of an elliptic curve, there is a finite \'etale cover $S'\to S$, given by the normalization of $S\times_CC'$ for some finite morphism $C'\to C$, such that 
\begin{align} \label{eq:tildeS'}
 S'\cong C'\times \tilde E .
\end{align}

The  cover $ S'\to S$ induces via $ X\to S$ a finite \'etale cover $ X'\to X$.
Then we get the following commutative diagram
$$
\xymatrix{
X'\ar[d]\ar[r]^-{h'} & S'\cong C'\times \tilde E \ar[r]\ar[d] & C'\times E\ar[d] \\
X\ar[r]^{h} & S \ar[r]& C\times E,
}
$$
where the composition of the lower horizontal arrow is $f\times g$ and $C'\times \tilde E\to C' \times E$ is the product of $\id_{C'}$ and the \'etale cover $\tilde E\to E$.
Since $h$ has connected fibres by construction, $h'$ has also connected fibres and so the Stein factorization of the natural map $X'\to E$ is given by the composition of $h'$ with the second projection $C'\times \tilde E\to \tilde E$. 
Hence, up to replacing $X$ by $X'$, $\tilde E\to E$ must be an isomorphism, as we want.
This concludes step 2.
\end{proof}

By steps 1 and 2, we may from now on assume that for general $c\in C$
\begin{align} \label{eq:X'_c:3}
X_c\cong E\times F_c ,
\end{align}
where $F_c$ is an elliptic curve that might depend on $c$ and the restriction of $g:X\to E$ to $X_c$ corresponds via the isomorphism in (\ref{eq:X'_c:3}) to the first projection $E\times F_c\to E$.

\medskip

\textbf{Step 3.}
$X$ is birational to $E\times S$ for a minimal surface $S$.

\begin{proof}
By (\ref{eq:X'_c:3}), we have for general $c\in C$ a splitting $X_c\cong E\times F_c$ for an elliptic curve $F_c$ which might depend on $c$.
To prove the claim in step 3, we need to show that such a splitting holds not only for a general point $c\in C$ but also for the generic point $\eta=\Spec \C(C)$ of $C$.
That is, we need to show that the generic fibre $X_{\eta}$ of $X\to C$ splits into a product of $E_{\C(C)}=E\times_\C \C(C)$ with an elliptic curve over $\C(C)$.

It is well-known that the splitting in (\ref{eq:X'_c:3}) implies that $X_{\eta}$  splits over the algebraic closure $\overline{\C(C)}$ into a product of $E_{\overline {\C(C)}}$ with an elliptic curve over $\overline{\C(C)}$.
The latter must be defined over $\C(C')$ for some finite cover $C'\to C$.
That is, there is a Zariski open non-empty subset $U\subset C$ and a finite Galois cover $U'\to U$ such that $X\times_CU'$ splits into $E\times T$, where $T$ is an elliptic surface over $U'$.
The base change $X\times_CU'$ fits into a diagram
$$
\xymatrix{
E\times T\cong X\times_CU'\ar[d]\ar[r] & E\times U'\ar[r]\ar[d] & U'\ar[d] \\
X\times_CU\ar[r] & E\times U \ar[r]& U  .
}
$$
Here the outer square is Cartesian.
The square on the right is Cartesian as well and so it follows that the square on the left must be Cartesian as well.
By assumptions, $U'\to U$ is Galois and we denote its Galois group by $G$.
Then, $G$ acts faithfully on $U'$ with $U'/G=U$.
By base change, $G$ also acts on $E\times U'$ and $X\times_CU'$ with quotients $E\times U$ and $X\times_CU$, respectively.
Moreover, the upper horizontal arrows in the above diagram are $G$-equivariant.
In particular, the map $E\times T\to E\times U'$ is $G$-equivariant and $G$ acts trivially on the first factor of $E\times U'$.
This implies that $G$ acts trivially on the first factor of $E\times T$ and so 
$$
X\times_CU\cong (E\times T)/G=E\times (T/G),
$$
where $T/G$ denotes the quotient of $T$ by the induced action of $G$.
This concludes step 3.
\end{proof}

By step 3, there is a minimal surface $S$ such that $X$ is birational to $E\times S$. 
Since $X$ has positive Kodaira dimension, the same holds for $S$ and so $K_S$ is nef.
Hence, $X$ and $E\times S$ are birational minimal models, 
and so they are connected by a sequence of flops, see \cite{kol89}.
Since $E\times S$ does not admit any non-trivial flop, $X\cong E\times S$, as we want.
This concludes Theorem \ref{thm:kod=0,1,2,3} in the case treated in Section \ref{subsec:X_c=curve,phi(F)=curve}.

\subsection{The general fibre of $f$ is via the Albanese map contracted to a point}
\label{subsec:X_c=curve,phi(F)=pt}

Since in the present case, the general fibre of $f:X\to C$ is contracted via $a$ to a point, every fibre must be contracted to a point and so the 
 Albanese map factors as $f:X\to C \to \Jac(C)\cong \Alb(X)$ .
Since $X$ carries a holomorphic one-form $\omega$ such that $(H^\ast(X,\C),\wedge \omega)$ is exact, $b_1(X)\neq 0$ and so $g(C)\geq 1$.
Moreover, $\omega$ is a pullback of a one-form from $C$ and so exactness of $\wedge \omega$ shows that $g(C)=1$.
Hence, $C\cong \Jac(C)\cong \Alb(X)$. 

By Lemmas \ref{lem:c1c2=c3=0} and \ref{lem:Grassi:c1c2=0}, the general fibre of $f$ is either an abelian surface or a bi-elliptic surface.

\textbf{Case 1.} $h^{2,0}(X)\neq 0$.

In this case, $X$ carries a nontrivial holomorphic two-form $\xi$.
Since $h^{1,0}(X)=1$, exactness of $(H^\ast(X,\C),\wedge \omega)$ shows that $\xi\wedge \omega\neq 0$.
Since $\omega$ is the pullback of a one-form on $C\cong \Alb(X)$, the condition $\xi\wedge \omega\neq 0$ then implies that $\xi$ is not vertical.
Hence, by Theorem \ref{thm:CP}, $f:X\to C$ is quasi-smooth with typical fibre $F$, a bi-elliptic or an abelian surface.
Since $\xi$ is not vertical with respect to $f$, it restricts to a nonzero form on the general fibre of $f$ and so $F$ must be an abelian surface.
The multiple fibres of $f$ give rise to an orbifold structure on $C$ which is good because $g(C)\geq 1$, see Theorem \ref{thm:orbifold-cover}.
Hence, Proposition \ref{prop:good-cover} shows that $X$ splits into a product after a finite \'etale cover, as we want. 
This proves Theorem \ref{thm:kod=0,1,2,3} in the case treated in Section \ref{subsec:X_c=curve,phi(F)=pt} if $h^{2,0}(X)\neq 0$.

\textbf{Case 2.} $h^{2,0}(X)=0$ and the general fibre of $f$ is an abelian surface.

In this case, consider the sheaf $f_\ast \omega_X$.
By Koll\'ar's theorem \cite[Theorem I.2.1]{kol86}, this is locally free of rank one, i.e.\ a line bundle on $C$.
Since $h^{3,1}(X)=h^{2,0}(X)$ vanishes, the Leray spectral sequence shows that
$$
H^1(C,f_\ast \omega_X)=0 .
$$
Moreover, $H^0(C,f_\ast \omega_X)=0$, as otherwise we get $h^{3,0}(X)\neq 0$, which is impossible because $h^{2,0}(X)=0$ and $\wedge \omega$ is exact on cohomology.
It follows that the line bundle $f_\ast \omega_X$ has no cohomology and so Riemann--Roch implies that it has degree zero.
That is, $f_\ast \omega_X\in \Pic^0(C)$.

Since $h^{3,0}(X)=0$, the cohomology support locus
$$
\{L\in \Pic^0(C)\mid H^0(X,\omega_X\otimes f^\ast L)\neq 0\}
$$ 
is a union of proper subtori of $\Pic^0(C)$, translated by torsion points \cite{Si}.
Since $\Pic^0(C)$ is an elliptic curve, the above set is in fact a union of torsion line bundles.
This set contains $L:=(f_\ast \omega_X)^{-1}\in \Pic^0(C)$, because
$$
H^0(X,\omega_X\otimes f^\ast L)=H^0(C,f_\ast \omega_X\otimes L)=H^0(C,\mathcal O_C)\neq 0 .
$$
Hence, $f_\ast \omega_X \in \Pic^0(C)$ is a torsion line bundle. 
But then up to passing from $X$ to a suitable  \'etale covering (induced by a \'etale covering of $C$), we may (by flat base change) assume that $f_\ast \omega_X$ is trivial.
This implies $h^{3,0}(X)\neq 0$ and so $h^{2,0}(X)\neq 0$ by exactness of $(H^\ast(X,\C),\wedge \omega)$.
Hence, we may conclude via Case 1.

\textbf{Case 3.} $h^{2,0}(X)=0$ and the general fibre of $f$ is a bi-elliptic surface.

In this case, $R^1f_\ast\omega_{X}=R^1f_\ast\omega_{X/C}$ is locally free of rank one.
Since $h^{2,0}(X)= h^{3,1}(X)=0$, $H^0(C,R^1f_\ast\omega_X)=0$ by the Leray spectral sequence (which degenerates by Koll\'ar's theorem or because $C$ is a curve).
Similarly, $b_1(X)=b_1(C)$ implies $H^0(C,R^1f_\ast\mathcal O_X)=0$.
By relative Serre duality, 
$$
 R^1f_\ast\mathcal O_X\cong (R^1f_\ast \omega_{X/C})^\ast\cong (R^1f_\ast \omega_{X})^\ast ,
 $$
 where we used $\omega_C=\mathcal O_C$ in the last step.
Hence, by Serre duality on $C$,
$$
0=H^0(C,R^1f_\ast\mathcal O_X)= H^1(C,R^1f_\ast \omega_{X}).
$$
That is, $R^1f_\ast \omega_X$ is a line bundle on $C$ without cohomology and so it must be in $\Pic^0(C)$.

Since $h^{3,1}(X)=0$, 
the cohomology support locus
$$
\{L\in \Pic^0(C)\mid H^1(X,\omega_X\otimes f^\ast L)\neq 0\}
$$ 
is a union of torsion line bundles (see \cite{Si}), where we use that $C$ is an elliptic curve.
As before, the Leray spectral sequence $E_2^{p,q}=H^p(C,R^qf_\ast \omega_X)\Rightarrow H^{p+q}(X,\C)$ degenerates at $E_2$.
Hence, 
$$ 
\{L\in \Pic^0(C)\mid H^0(C,R^1f_\ast \omega_X\otimes L)\neq 0\}
$$
is also a union of torsion line bundles.
Since $R^1f_\ast \omega_X$ is a line bundle of degree zero on $C$, we deduce that it must be torsion. 
Hence, after a suitable \'etale base change $C'\to C$, we may (by flat base change)
 assume that $R^1f_\ast \omega_X\cong \mathcal O_C$.
This implies $h^{2,0}(X)=h^{1,3}(X)\neq 0$, and so we are done by Case 1, treated above.

Cases 1, 2 and 3 above finish the proof of Theorem \ref{thm:kod=0,1,2,3} in the case treated in Section \ref{subsec:X_c=curve,phi(F)=pt}.
Together with the results in Sections \ref{subsec:X_c=curve,phi(F)=surface} and \ref{subsec:X_c=curve,phi(F)=curve}, this finishes the proof of Theorem \ref{thm:kod=0,1,2,3} in the case of Kodaira dimension one.

\section{Proof of Theorem \texorpdfstring{\ref{thm:kod=0,1,2,3}}{4.1}} \label{sec:proof:thm:kod=0,1,2,3}


Let $X$ be a smooth projective minimal threefold which satisfies (\ref{item:exact}) as in Theorem \ref{thm:kod=0,1,2,3}.
By Corollary \ref{cor:Xnot_gentype} $X$ cannot be of general type, and so we are left with the cases $\kappa(X)=0,1,2$.

If $\kappa(X)=2$, then Theorem \ref{thm:kod=0,1,2,3} is proven in Section \ref{subsec:kappa=2}.

If $\kappa(X)=1$, then the result is proven in Sections  \ref{subsec:X_c=curve,phi(F)=surface}, \ref{subsec:X_c=curve,phi(F)=curve} and \ref{subsec:X_c=curve,phi(F)=pt} above.

If $\kappa(X)=0$, $mK_X$ is trivial for some $m>0$ by the abundance conjecture for complex projective threefolds, see \cite{Ka92}. 
It then follows from the Beauville--Bogomolov decomposition theorem (see \cite{bea83}) and the fact that $b_1(X)>0$ that some finite \'etale cover  $\tau:X'\to X$ splits into a product $A\times S$ with a positive-dimensional abelian factor $A$, as required.  

This concludes the proof of Theorem \ref{thm:kod=0,1,2,3}.

\section{1-forms on Mori fiber spaces}

In this Section we prove Theorems \ref{thm:classification} and \ref{thm:cup_omega_exact} in the case of negative Kodaira dimension.
The main results are Theorems \ref{thm:conic-bundle} and \ref{thm:delPezzo} below.

\subsection{Preliminaries}

Recall the following well-known lemma.

\begin{lemma}\label{lem:euler}
Let $f:X\rightarrow E$ be a surjective morphism from a compact complex manifold $X$ to an elliptic curve $E$. 
Then the topological Euler characteristic of $X$ is given by
$$
e(X)=\sum_{s\in E}(e(X_s)-e(X_g)),
$$ 
where $X_g$ denotes a fixed smooth fibre of $f$.
\end{lemma}


\begin{lemma}\label{lem:star-for-surface}
Let $f: X\rightarrow S$ be a surjective morphism from a smooth complex projective threefold $X$ to a smooth projective surface $S$ with $f^*: H^1(S, \C)\cong H^1(X, \C)$. 
Suppose that for $\omega\in H^0(S, \Omega_S^1)$, $(H^{\ast}(X,\C),\wedge  f^*\omega)$ is exact. Then $(H^{\ast}(S,\C),\wedge \omega)$ is exact.
\end{lemma}

\begin{proof}
Assume first that there is a divisorial contraction $S\rightarrow \overline S$ with $\overline S$ smooth.
Then the composition $g: X \rightarrow \overline S$ contracts a divisor $D$ to a point, contradicting Lemma \ref{lem:contracted-divisor}, applied to the Albanese morphism $a:X\to \Alb(X)$, because $a$ factors through $S$ by assumptions. 
Hence, a contraction as above does not exist and so $S$ is minimal. 
Next, we consider the following diagram
$$
\xymatrix{
  \C\ar[r] & H^1(X, \C)\ar[r]^-{\wedge f^*\omega}& H^2(X, \C)\ar[r]^-{\wedge f^*\omega} & H^3(X, \C)\ar[r]^-{\wedge f^*\omega}&H^4(X, \C)\\
 \C \ar[r] & H^1(S, \C)\ar[u]^{\cong}_{f^*}\ar[r]^-{\wedge \omega}& H^2(S, \C)\ar@{^{(}->}[u]_{f^*}\ar[r]^-{\wedge \omega} & H^3(S, \C)\ar@{^{(}->}[u]_{f^*}\ar[r]^-{\wedge \omega}&  H^4(S,\C)=\C\ar@{^{(}->}[u]_{f^*}.}
$$
By diagram chasing, we have that $(H^{\ast}(S,\C),\wedge \omega)$ is exact at $H^i(S,\C)$ for $i=0, 1, 2$. 
Also, $(H^{\ast}(S,\C),\wedge \omega)$ is exact at $H^4(S,\C)$.  
Assume $(H^{\ast}(S,\C),\wedge \omega)$ is not exact at $H^3(S,\C)$. 
Then $c_2(S)<0$ and so  $S$ is a ruled surfaces over a curve of genus $>1$, see \cite[Theorem VI.1.1]{barth-etal}.
The latter contradicts the fact that the Albanese image of $X$ (and hence that of $S$) is fibered by tori by item (\ref{item:thm:fibred}) in Theorem \ref{thm:C=>...}. 
\end{proof}

%

\subsection{Conic bundles over surfaces}

\begin{theorem} \label{thm:conic-bundle}
Let $X$ be a smooth projective threefold which admits the structure of a Mori fibre space  $f:X\to S$ over a projective surface $S$.
Let $\omega\in H^0(X, \Omega^1_X)$ be a holomorphic one-form on $X$ such that for any \'etale cover $\tau:X'\to X$, $(H^{\ast}(X',\C),\wedge \tau^\ast\omega)$ is exact. 
Then the following holds:
\begin{enumerate}
\item $S$ is a smooth projective surface which satisfies (\ref{item:holo}) $\Leftrightarrow$ (\ref{item:closed}) $ \Leftrightarrow$ (\ref{item:exact});\label{item:thm:conic:S} 
\item there is a smooth map $\pi:X\to A$ to a positive-dimensional abelian variety $A$;\label{item:thm:conic:pi}
\item $f$ is either smooth, or $A$ is an elliptic curve and the degeneration locus of $f$ is a disjoint union of smooth elliptic curves which are \'etale over $A$ (via the map $S\to A$ induced by $\pi$);\label{item:thm:conic:Delta_f} 
\item $\omega$ has no zero on $X$. \label{item:thm:conic:omega}
\end{enumerate}  
\end{theorem}

\begin{proof}
By \cite[Theorem 3.5]{mori-annals}, $f$ is a conic bundle of relative Picard rank one, $S$ is smooth and the discriminant locus $\Delta_f$ of $f:X\to S$ is a curve with at worst ordinary double points.

Since $f^\ast$ induces an isomorphism on $H^1$, $\Alb(X)\cong \Alb(S)$ and so we get a commutative diagram
$$
\xymatrix{
   X\ar[dr]_-{\alpha} \ar[r]^-f &S\ar[d]^-{\beta}\\
   & A,}
$$
where $A=\Alb(X)=\Alb(S)$ and $\alpha,\beta$ are the respective Albanese morphisms.
Hence there is a one-form $\gamma\in H^0(A,\Omega_A^1)$ with $\omega=\alpha^\ast \gamma=f^\ast \beta^\ast \gamma$.

Let $\tau:S'\to S$ be a finite \'etale cover, and let $X':=X\times_SS'$ be the induced \'etale cover of $X$.
Then $X'\to S'$ is a conic bundle and so $b_1(X')=b_1(S')$.
Hence, Lemma \ref{lem:star-for-surface} implies that $(H^\ast(S',\C),\wedge \tau^\ast \beta^\ast\gamma)$ is exact.
This proves item (\ref{item:thm:conic:S}).
In particular, $S$ admits a smooth morphism to an elliptic curve or abelian surface by \cite[Corollary 3.2]{Sch19}, which proves item (\ref{item:thm:conic:pi}) in the case where $f$ is smooth.
For the remainder of the proof of items (\ref{item:thm:conic:pi}) and (\ref{item:thm:conic:Delta_f}), we may thus assume that $\Delta_f\neq \emptyset$.

By (\ref{item:thm:conic:S}) and \cite[Corollary 3.1]{Sch19}, $S$ is one of the following:
\begin{enumerate}[(a)]
\item a minimal ruled surface over an elliptic curve;\label{item:S=ruled}
\item an abelian surface;\label{item:S=abelian} 
\item a minimal elliptic surface $h:S\to C$ such that one of the following holds:
\begin{enumerate}[(i)]
\item $h$ is smooth, $C$ is an elliptic curve and $\beta^\ast \gamma \in h^\ast H^0(C,\Omega_C^1)$; \label{item:f-smooth}
\item $h$ is quasi-smooth, i.e.\ all singular fibres are multiple fibres, and the restriction of $\beta^\ast \gamma$ to a general fibre of $h$ is non-zero. \label{item:f-quasi-smooth}
\end{enumerate}
\end{enumerate} 
By \cite[lemma 7.1.10]{pash99}, 
$$ 
e(X)=2(e(S)-p_a(\Delta_f)+1),
$$
where $p_a$ denotes the arithmetic genus of $\Delta_f$.
Since $e(X)=0$ and $e(S)=0$ by exactness of  $(H^\ast(X,\C),\wedge  \alpha^\ast\gamma)$ and $(H^\ast(S,\C),\wedge  \beta^\ast\gamma)$, we deduce
$$
p_a(\Delta_f)=1.
$$
Since $\Delta_f$ is a nodal curve, this implies that each connected component of $\Delta_f$ is either a smooth elliptic curve with some attached trees of rational curves, or a rational curve with a single node with some attached trees of rational curves. 
Note that $\Delta_f$ cannot contain a smooth $\CP^1$ which is attached to the remaining components at only one point $p$, as the monodromy of the lines above $\CP^1\setminus\{p\}\cong \mathbb A^1$ would need to be trivial and so $f$ could not have relative Picard rank one, contradicting our assumptions.
Moreover, since none of the surfaces in (\ref{item:S=ruled})--(\ref{item:f-quasi-smooth}) contains a rational curve with at least one node, we conclude that $\Delta_f$ must be a disjoint union of smooth elliptic curves. 
By assumptions, $\Delta_f\neq \emptyset$ and so we pick an irreducible component $D\subset \Delta_f$, which is automatically a smooth connected component of $\Delta_f$.
Hence, each fibre of $f$ above a point $d\in D$ is given by two distinct lines which meet in a point, see e.g.\ \cite[Proposition 7.1.8(i)]{pash99}.

By \cite[Corollary 3.2]{Sch19}, some \'etale cover $S'\to S$ is either a simple abelian surface, or it splits into the product of an elliptic curve with another curve.
This implies that there is a finite \'etale cover $\tau:S'\to S$ whose restriction to $D$ induces an \'etale cover that trivializes the monodromy of the two lines above points of $D$ in the conic bundle $f:X\to S$.
Let $X'\to X$ be the finite \'etale cover induced by $S'$.
Then, $f':X'\to S'$  has relative Picard rank at least two, because the monodromy of the two lines above $D':=\tau^{-1}(D)$ is trivial.
In particular, $X'$ admits a divisorial contraction where exactly one of the two families of lines above $D'$ is contracted, and we exhibit $X'$ as the blow-up $X'\cong Bl_{D'}Y$ of a conic bundle $Y\to S'$ which is smooth above $D'\subset S'$.
By Proposition \ref{prop:MMP}, the pullback $\tau^\ast \beta^\ast \gamma$ restricts nontrivially on $D'$.
This implies that $\beta^\ast \gamma$ restricts nontrivially on $D$.
Hence, there is a surjective morphism $p:A\twoheadrightarrow E$ to an elliptic curve $E$, so that the composition $p\circ \beta: S\to E$ restricts on $D$ to an isogeny $D\to E$.

Since $S$ splits into a product after some \'etale cover by \cite[Corollar 3.2(e)]{Sch19}, one checks that $p\circ \beta: S\to E$ must be a smooth morphism.
Since the remaining components of $\Delta_f$ are all smooth elliptic curves that are disjoint from $D$, we conclude that in fact any component of $\Delta_f$ is \'etale over $E$ via the morphism $p\circ \beta$.
This concludes item (\ref{item:thm:conic:Delta_f}).
Moreover, it shows that the composition $p\circ \beta\circ f:X\to E$ is smooth, which proves item (\ref{item:thm:conic:pi}).

It remains to prove item (\ref{item:thm:conic:omega}).
We have seen that the one-form $\omega=f^\ast \beta^\ast \gamma$ comes from a form $\beta^\ast \gamma$  that has no zeros on $S$ by \cite[Corollary 3.1]{Sch19}.
Hence, $\omega$ has no zeros on $X$ in the case where $f$ is smooth.
If $f$ is not smooth, we have seen above that $\beta^\ast \gamma$ restricts nontrivially on each component of the ramification locus $\Delta_f$, and this implies that $\omega$ has no zeros along $f^{-1}(D)$, and hence no zeros on $X$.
This proves item (\ref{item:thm:conic:omega}), which finishes the proof of Theorem \ref{thm:conic-bundle}.
\end{proof}

\subsection{Del Pezzo fibrations}

\begin{theorem} \label{thm:delPezzo}
Let $X$ be a smooth complex projective threefold which admits a holomorphic one-form $\omega$ such that for any \'etale cover $\tau:X'\to X$, $(H^{\ast}(X',\C),\wedge \tau^\ast \omega)$ is exact.
Suppose that $X$ admits the structure of a Mori fibre space $f:X\to C$ over a curve $C$.
Then $C$ is an elliptic curve and $f$ is smooth.
In particular, $\omega$ has no zeros on $X$.
\end{theorem}

\begin{proof}
By \cite[Theorem 3.5.2]{mori-annals}, $C$ is a smooth projective curve and any fibre of $f$ is an irreducible reduced del Pezzo surface which is Gorenstein, as its canonical bundle is given by the restriction of the canonical bundle of $X$.
Moreover, since $\omega$ has its zeros exactly at the singular points of the fibres of $f$, and since it cannot have isolated zeros by item (\ref{item:thm:Z(omega)}) in Theorem \ref{thm:C=>...}, the singular fibres of $f$ are reduced and non-normal.

Note that $f^\ast$ induces an isomorphism $H^{1}(C)\cong H^1(X)$.
Since $(H^{\ast}(X,\C),\wedge \omega)$ is exact, we deduce that $C$ is an elliptic curve, see e.g.\ item (2) in Theorem \ref{thm:C=>...}. 
It remains to prove that $f$ is smooth.

Since $\rho(X/C)=1$ and $h^{2,0}(X)=0$, $b_2(X)=b_2(C)+1=2$.
Since $b_1(X)=b_1(C)=2$ and $(H^{\ast}(X,\C),\wedge \omega)$ is exact, we then conclude
$$
b_1(X)=b_5(X)=2,\ \ b_2(X)=b_4(X)=2,\ \ \text{and} \ \ b_3(X)=2 .
$$

We consider the Leray spectral sequence
$$
E_2^{p,q}=H^p(C,R^qf_\ast \C)\Longrightarrow H^{p+q}(X,\C).
$$
Since $C$ is a smooth projective curve, this spectral sequence degenerates at $E_2$, see e.g.\ \cite[Theorem 4.24]{peters-steenbrink}.
Since the general fibre of $f$ is a smooth del Pezzo surface, $R^1f_\ast \C$ is a skyscraper sheaf, which implies $E_2^{1,1}=0$. 
We also have 
$
E_2^{2,0}=H^2(C,\C)\cong \C$.
Since $b_2(X)=2$ and $E_2=E_\infty$, we conclude
\begin{align}\label{eq:delpezzo:H^0R^2}
H^0(C,R^2f_\ast \C)=E_2^{0,2}=E_\infty^{0,2}\cong \C .
\end{align}
In particular, $H^0(C,R^2f_\ast \C)$ is generated by the section of $R^2f_\ast \C$ induced by an ample class on $X$. 

By topological proper base change, we have for any $c\in C$ and any $i\in \N$:
$$
(R^if_\ast \C)_c\cong H^i(X_c,\C).
$$
We claim that this implies that $b_2(X_c)$ is bounded from above by the second Betti number of a smooth fibre of $f$.
Indeed, if not, then there is a neighbourhood $U$ of $c\in C$ and a section $s\in R^2 f_\ast \C(U)$ which is nonzero at the stalk at $c$ but zero at all other stalks.
But then after taking the extension by zero, $s$ can be extended to a nontrivial global section of $R^2f_\ast \C$ which is zero away from the point $c$, contradicting (\ref{eq:delpezzo:H^0R^2}).

Let $t\in C$ be a general point, so that the fibre $X_t$ is smooth.
For any $c\in C$, we have seen above that $b_2(X_c)\leq b_2(X_t)$.
Moreover, since $X_c$ is irreducible, $b_0(X_c)=b_4(X_c)=1$ and so the topological Euler characteristics satisfy 
$$
e(X_c)\leq  e(X_t)
$$
for all $c\in C$.
Since $C$ is an elliptic curve and $X$ has trivial topological Euler characteristic (cf.\ Lemma \ref{lem:c1c2=c3=0}), we deduce from Lemma \ref{lem:euler} that 
$$
e(X_c)=  e(X_t)
$$
for all $c\in C$.
In particular, 
$$
b_2(X_c)=b_2(X_t)
$$
for all $c\in C$ and so the stalks of $R^2f_\ast \C$ have the same dimension at all points of $C$.
This implies that $R^2f_\ast \C$ is a local system, because we have seen above that any local section of $R^2f_\ast \C$ vanishes if it vanishes generically.  

We claim that in order to prove that $f:X\to C$ is smooth, we may replace $X$ by any finite \'etale cover, induced by a finite \'etale cover $C'\to C$.
To prove this, let $\tau:X'\to X$ by a finite \'etale cover, induced by a finite \'etale cover $C'\to C$, and assume that $f':X'\to C'$ is smooth.
Then the natural map $X'\to C$ is smooth as well and since $X'\to X$ is \'etale, $f$ is smooth, as claimed. 

Since $R^2f_\ast \C$ is generically a polarized variation of Hodge structure of type $(1,1)$, its monodromy representation is finite.
This implies that up to replacing $X$ by a finite \'etale covering, we may assume that $R^2f_\ast \C$ is trivial.
Applying Proposition \ref{prop:MMP} to this base change repeatedly, we may assume that we again arrive at the situation of a Mori fibre space over $C$ and so $\rho(X/C)=1$. 
Since $R^2f_\ast \C$ is trivial, $\rho(X/C)=1$ implies that it has rank one and so the general fibre of $f$ must be $\CP^2$ (because it is a smooth del Pezzo surface with $b_2=1$).
It is not hard to prove directly that this implies that $f$ cannot have any non-normal fibres; alternatively, this claim also follows from \cite[Theorem 3.1]{Fuj-delPezzo}.
Since we have seen above that all singular fibres of $f$ are non-normal, this proves that $f$ is smooth. 
This finishes the proof of Theorem \ref{thm:delPezzo}.
\end{proof}

\section{Summary of the argument} \label{sec:proofs}

\begin{proof}[Proof of Theorems \ref{thm:A=B=C} and \ref{thm:classification}]
Since \ref{thm:classification} $\Rightarrow$ \ref{thm:A=B=C}, it suffices to prove Theorem \ref{thm:classification}.
Since (\ref{item:classification:MMP}) $\Rightarrow$ (\ref{item:holo}) $\Rightarrow$ (\ref{item:closed}) $\Rightarrow$ (\ref{item:B'}') is clear and (\ref{item:B'}') $\Rightarrow$ (\ref{item:exact}) is proven in the appendix (see Theorem \ref{thm:B'}), it suffices to prove (\ref{item:exact}) $\Rightarrow$  (\ref{item:classification:MMP}) for smooth projective threefolds.
For this, let $X$ be a smooth projective threefold and assume that (\ref{item:exact}) holds for the holomorphic one-form $\omega$ on $X$.

By Corollary \ref{cor:MMP}, the minimal model program for $X$ yields a smooth projective threefold $X^{\min}$ and a proper birational morphism $\sigma:X\to X^{\min}$ which is a sequence of blow-ups along elliptic curves which are not contracted by the natural map to $\Alb(X^{\min})$.
This proves item (\ref{item:thm:classification:sigma}). 

To prove items (\ref{item:thm:classification:pi}), (\ref{item:thm:classification:kappa=0,1,2,3}) and (\ref{item:thm:classification:kappa=-infty}), we may replace $X$ by $X^{\min}$ and assume that $X$ is either minimal or a Mori fibre space.
The assertions then follow from 
 Theorem \ref{thm:kod=0,1,2,3}, Corollary \ref{cor:kod=0,1,2,3}, and Theorems \ref{thm:conic-bundle} and \ref{thm:delPezzo}. 

This concludes the proof of Theorems \ref{thm:A=B=C} and \ref{thm:classification}.
\end{proof}

\begin{proof}[Proof of Theorem \ref{thm:cup_omega_exact}]
By Proposition \ref{prop:MMP}, it suffices to prove Theorem \ref{thm:cup_omega_exact} for a smooth projective threefold $X$ which is either minimal or a Mori fibre space.
In the latter case, Theorem \ref{thm:cup_omega_exact} follows from Theorems \ref{thm:conic-bundle} and \ref{thm:delPezzo}; in the former case, Theorem \ref{thm:cup_omega_exact} follows from 
item (\ref{item:omega}) in Corollary \ref{cor:kod=0,1,2,3}. 
This concludes the proof of Theorem \ref{thm:cup_omega_exact}.
\end{proof}

\begin{proof}[Proof of Corollary \ref{cor:fibre-bundle}]
Let $f:X\to A$ be a smooth morphism from a smooth projective threefold $X$ with $\kappa(X)\geq 0$ to an abelian variety $A$.
If $A$ is zero-dimensional, the claim is trivial and so we may assume that $\dim A>0$.
Hence, $X$ carries a holomorphic one-form without zeros and so the equivalent conditions (\ref{item:holo}) $\Leftrightarrow$ (\ref{item:closed}) $ \Leftrightarrow$ (\ref{item:B'}')  $\Leftrightarrow$ (\ref{item:exact}) $\Leftrightarrow$ (\ref{item:classification:MMP}) hold by Theorem \ref{thm:classification}.

If $X$ is not minimal, then by Proposition \ref{prop:MMP}, $X$ is the blow-up $Bl_CY$ of a smooth projective threefold $Y$ along an elliptic curve $C\subset Y$. 
Since $f:X\to A$ factors through $Y$, this implies that $A$ must be an elliptic curve and $C$ must be \'etale over $A$ ($f$ is non-flat if $\dim A=3$, and it has singular fibres if $\dim A=2$ or if $\dim A=1$ and $C$ is not \'etale over $A$). 
Hence, up to replacing $X$ by $Y$, we may inductively assume that $X$ is minimal.
We claim that under this assumption, $f$ is an analytic fibre bundle.
To prove this, we may replace $X$ by any finite \'etale cover.
Hence, by (\ref{item:thm:classification:kappa=0,1,2,3}), we may assume that $X=A'\times S'$ is a product of an abelian variety $A'$ of positive dimension and a smooth projective variety $S'$.
Then for any $s\in S'$, $A'\times \{s\}$ maps to a translate of a fixed abelian subvariety $A''\subset A$.
Let $B:=A/A''$.
If $\dim B=0$, then we are done.
Otherwise, for any $a\in A'$, consider the natural composition 
$$
S'=\{a\}\times S'\hookrightarrow A'\times S' =X \stackrel{f}\longrightarrow A\twoheadrightarrow A/A''=B ,
$$
which is smooth, because $f$ is smooth.
Since $\dim B>0$, \cite[Corollary 3.2]{Sch19} implies that up to a further \'etale covering, $S'$ splits of a positive-dimensional abelian variety as a direct factor.
Arguing as before therefore concludes the proof of the corollary by induction on the dimension.
\end{proof}

\appendix

\section{}

The purpose of this appendix is to show that the arguments in \cite{Sch19} easily generalizes to prove the following, which is inspired by questions of Kieran Kedlaya and Burt Totaro.

\begin{theorem} \label{thm:B'}
Let $X$ be a compact connected K\"ahler manifold.
Then (\ref{item:B'}') $\Rightarrow$ (\ref{item:exact}).
\end{theorem}

\begin{proof}
Let $X$ be a compact connected K\"ahler manifold.
In order to show that (\ref{item:exact}) holds, it thus suffices by the argument in \cite[Section 2.3]{Sch19} to show that there is a local system $L$ with stalk $\C$ on $X$ whose first Chern class is trivial and such that $H^i(X,L)=0$ for all $i$.
Since the space of such local systems is given by $\Hom(\pi_1(X),\C^\ast)$, it depends only on the homotopy type of $X$.
Moreover, the corresponding cohomology groups do not change when we replace $X$ by a homotopy equivalent CW complex, see e.g.\ \cite[VI.2.6*]{whitehead}.
Hence, to prove (\ref{item:B'}') $\Rightarrow$ (\ref{item:exact}), it suffices to show that any CW complex $Y$ which admits a finite $\Q$-homology fibration $f:Y\to S^1$ carries a local system with stalk $\C$, trivial first Chern class and no cohomology.
Following a suggestion of Botong Wang (see \cite[Remark 2.3]{Sch19}), we claim that the local system $L=f^\ast L_\lambda$ on $Y$ has this property, where $L_\lambda$ is a generic local system on $S^1$ with monodromy given by a general element $\lambda \in \C$.
To show this claim, consider the Leray spectral sequence with $E_2$-term
$$
E_2^{p,q}=H^p(S^1,R^qf_\ast f^\ast L_\lambda)=H^p(S^1,L_\lambda\otimes R^qf_\ast \C) \Longrightarrow H^{p+q}(Y,L) .
$$
By our assumptions, the sheaf $R^qf_\ast \C$ is a local system on $S^1$ with stalk a finite dimensional $\C$-vector space $V^q$.
Since $\lambda\in \C$ is general, $\lambda^{-1}$ is different from all eigenvalues of the natural monodromy operator on $V^q$.
Hence, $L_\lambda\otimes R^qf_\ast \C$ is a local system on $S^1$ which corresponds to a finite-dimensional $\C$-vector space together with a monodromy operator whose eigenvalues are all different from one.
In particular, $H^0(S^1,L_\lambda\otimes R^qf_\ast \C)=0$.
Since the Euler characteristic of any local system of finite rank on $S^1$ is zero, we also get $H^1(S^1,L_\lambda\otimes R^qf_\ast \C)=0$.
Hence, $E_2^{p,q}=0$ for all $p,q$ and so $H^k(X,L)=0$ for all $k$.
This concludes the proof.
\end{proof}

\section*{Acknowledgement}
The second author thanks Dieter Kotschick for sending him the preprint \cite{Ko} in spring 2013, where he poses the problem about the equivalence of (\ref{item:holo}) and (\ref{item:closed}).
We are grateful for useful comments and conversations to Antonella Grassi, Thomas Peternell, Mihnea Popa, Christian Schnell and Burt Totaro.
Both authors are supported by the DFG grant ``Topologische Eigenschaften von Algebraischen Variet\"aten'' (project no.\ 416054549).

\end{document}